\documentclass[12pt, colorinlistoftodos]{amsart}
\usepackage{graphicx}
\usepackage{amsmath, enumerate}
\usepackage{amsfonts, amssymb, amsthm, xcolor, stmaryrd}
\usepackage{tikz, tikz-cd, caption, tabu, longtable, mathrsfs, xtab, centernot, mathtools}
\usepackage{todonotes}
\usepackage{float}
\usepackage[pagebackref=true]{hyperref}

\numberwithin{equation}{section}
\theoremstyle{plain}

\newtheorem{theorem}{Theorem}[section]
\newtheorem{lemma}[theorem]{Lemma}
\newtheorem{proposition}[theorem]{Proposition}
\newtheorem{corollary}[theorem]{Corollary}

\theoremstyle{definition}

\newtheorem{conjecture}[theorem]{Conjecture}
\newtheorem{example}[theorem]{Example}

\theoremstyle{remark}
\newtheorem{remark}[theorem]{Remark}

\renewcommand{\O}{\mathcal{O}}

\newcommand{\R}{\mathbb{R}}

\newcommand{\Z}{\mathbb{Z}}
\newcommand{\Q}{\mathbb{Q}}
\newcommand{\C}{\mathbb{C}}
\newcommand{\D}{\mathbb{D}}
\newcommand{\A}{\mathbb{A}}

\newcommand{\reg}{{\mathrm{reg}}}
\renewcommand{\H}{\mathbb{H}}
\newcommand{\Mp}{{\mathrm{Mp}}}
\newcommand{\SL}{{\mathrm{SL}}}

\newcommand{\GL}{{\mathrm{GL}}}
\newcommand{\SO}{{\mathrm{SO}}}

\newcommand{\Oc}{\mathcal{O}}

\newcommand{\tr}{\mathrm{tr}}

\newcommand{\Res}{\mathrm{Res}}
\newcommand{\Nm}{\mathrm{Nm}}

\newcommand{\af}{\mathfrak{a}}
\newcommand{\mf}{\mathfrak{m}}
\newcommand{\Cl}{\mathrm{Cl}}
 \renewcommand{\k}{E}

\renewcommand{\Re}{\mathrm{Re}}
\renewcommand{\Im}{\mathrm{Im}}
\newcommand{\Gal}{\mathrm{Gal}}

\newcommand{\new}{\mathrm{new}}
\newcommand{\GSpin}{\mathrm{GSpin}}

\newcommand{\CT}{\mathrm{CT}}
\newcommand{\Sc}{\mathcal{S}}

\newcommand{\Rb}{\mathbb{R}}
\newcommand{\Zb}{\mathbb{Z}}
\newcommand{\Qb}{\mathbb{Q}}
\newcommand{\Cb}{\mathbb{C}}

\newcommand{\ebf}{{\mathbf{e}}}

\newcommand{\Hb}{\mathbb{H}}

\newcommand{\Nb}{\mathbb{N}}
\newcommand{\smat}[4]{\left(\begin{smallmatrix}
                 #1 & #2\\
                 #3 & #4
\end{smallmatrix}\right)}

\newcommand{\lp}{\left (}
\newcommand{\rp}{\right )}

\newcommand{\Ab}{\mathbb{A}}
\newcommand{\Db}{\mathbb{D}}

\newcommand{\f}{\mathrm{f}}
\newcommand{\vol}{\mathrm{Vol}}
\newcommand{\cha}{\mathrm{char}}
\newcommand{\Pet}{\mathrm{Pet}}
\newcommand{\ab}{\mathrm{ab}}
\newcommand{\Aut}{\mathrm{Aut}}

\newcommand{\Gm}{\mathbb{G}_m}
\newcommand{\bfrak}{\mathfrak{b}}
\newcommand{\df}{\mathfrak{d}}
\newcommand{\pf}{\mathfrak{p}}
\newcommand{\Qbar}{{\overline{\mathbb{Q}}}}
\renewcommand{\c}{\mathrm{c}}
\newcommand{\GQ}{\mathrm{Gal}(\overline{\mathbb{Q}}/\mathbb{Q})}
\newcommand{\cf}{\mathfrak{c}}
\newcommand{\vr}{\varrho}
\newcommand{\vt}{\vartheta}
\newcommand{\prin}{\mathrm{prin}}

\begin{document}
\title[Harmonic Maass forms associated with CM newforms]{Harmonic Maass forms associated with CM newforms}
\author{Stephan Ehlen, Yingkun Li and Markus Schwagenscheidt}

\address{Mathematisches Institut, University of Cologne, Weyertal 86-90, D-50931 Cologne, Germany}
\email{stephan.ehlen@math.uni-koeln.de}

\address{Fachbereich Mathematik,
Technische Universit\"at Darmstadt, Schlossgartenstrasse 7, D--64289
Darmstadt, Germany}
\email{li@mathematik.tu-darmstadt.de}

\address{ETH Z\"urich Mathematics Dept., R\"amistrasse 101, CH-8092 Z\"urich, Switzerland}
\email{mschwagen@ethz.ch}
\date{\today}

\begin{abstract}
In this paper, we use a regularized theta lifting to construct harmonic Maass forms corresponding to binary theta functions of weight $k \ge 2$ under the $\xi$-operator. As a result, we show that their holomorphic parts have algebraic Fourier coefficients, with compatible Galois action. As an application, we prove rationality properties of coefficients of harmonic Maaass forms corresponding to CM newforms, answering a question of Bruinier, Ono and Rhoades.\end{abstract}

\maketitle

 \makeatletter
 \providecommand\@dotsep{5}
 \def\listtodoname{List of Todos}
 \def\listoftodos{\@starttoc{tdo}\listtodoname}
 \makeatother

\tableofcontents

\section{Introduction}
Harmonic Maass forms, which are non-holomorphic generalizations of classical holomorphic modular forms, have been introduced by Bruinier and Funke \cite{bruinierfunke04} in connection with theta liftings, but also played a crucial role in Zwegers' investigation of Ramanujan's mock theta functions \cite{zwegers}. Since then, there has been a lot of interest in harmonic Maass forms, in particular concerning the Fourier coefficients of their holomorphic parts. For example, it turns out that several natural generating series, e.g. of traces of CM values or cycle integrals of modular functions, appear as holomorphic parts of harmonic Maass forms \cite{bruinierfunke06, dit}. Moreover, it seems that the transcendency of these holomorphic part coefficients is a deep question. For instance, in some cases their transcendency controls the non-vanishing of central values of the first derivative of $L$-functions of certain newforms \cite{bruinieronoheegnerdivisors}.

For a given harmonic Maass form, there are only two cases when all its holomorphic part Fourier coefficients are known to be algebraic, namely when its shadow is a linear combination of unary or higher weight binary theta series.
The first case contains the modular completions of the generating series of Hurwitz class numbers \cite{HZ76} and
the mock theta functions of Ramanujan studied by Zwegers \cite{zwegers}.
The second case includes newforms with complex multiplication \cite{bruinieronorhoades}, which we describe now.

Let $\k = \Q(\sqrt{-D}) \subset \Qbar \subset \Cb$ be an imaginary quadratic field.
Given an algebraic Hecke character $\vr$ of weight $k \ge 2$
and modulus $\mf \subset \Oc_\k$, one can associate 
a newform $\vartheta_\vr \in S_{k}^\new(N,\chi)$ of weight $k$, level $N = D \Nm(\mf)$ with some Nebentypus $\chi$ (see \eqref{eq:vartheta}).
In \cite{bruinieronorhoades}, Bruinier, Ono and Rhoades studied the algebraicity properties of a harmonic Maass form $\tilde\vt_\vr \in H_{2-k}(N,\overline{\chi})$ such that $\xi_{2-k}(\tilde\vt_\vr) = \vt_\vr/\|\vartheta_\vr\|^2_\Pet$.
Though such a harmonic Maass form is not unique, they chose a so-called \textit{good} one, whose holomorphic part $\tilde \vt^+_\vr$ satisfies the following additional conditions.
\begin{enumerate}
\item $\tilde \vt^+_\vr$ is holomorphic at all cusps not equivalent to $\infty$.
  \item
  The principal part of $\tilde\vt^+_\vr$ at $\infty$ has coefficients in $F_\vr$.
\end{enumerate}
Here $F_\vr$ is the field of definition of $\vt_\vr$.
It is not hard to see that such good forms exist, and are unique up to addition of weakly holomorphic forms in $M^{!}_{2-k}(N, \overline \chi)$ which have coefficients in $F_\vr$ and are holomorphic away from the cusp $\infty$. 
Using indirect methods such as twisting and Hecke operators, Bruinier, Ono and Rhoades proved that all holomorphic part Fourier coefficients of good harmonic Maass forms for $\vartheta_\vr$
are in $F_\vr(\zeta_{ND})$, where $\zeta_{ND}$ denotes a primitive $(ND)$-th root of unity. 
In addition, they raised two questions (see Remarks 3(v) and 3(vi) loc.\ cit.)
\begin{enumerate}
\item[3(v)] Find an explicit construction of good $\tilde \vt_\vr$.
\item[3(vi)] Is there a good $\tilde \vt_\vr$ such that the Fourier coefficients of its holomorphic part are contained in $F_\vr$?
\end{enumerate}
In this paper, we answer the two questions above simultaneously.
In particular, we will extend the action of $\GQ$ from CM newforms to their $\xi$-preimages in a compatible way.
In terms of the Galois actions on Laurent series and Hecke characters  given in Section~\ref{sec:prelim}, we have the following result.

\begin{theorem}
  \label{thm:Gal}
  For every Hecke character $\vr$ of weight $k \ge 2$, there exists $\tilde\vt_\vr \in H_{2-k}(N,\overline{\chi})$ good for $\vt_\vr$ such that
  its holomorphic part $\tilde \vt_\vr^+$ satisfies
  \begin{equation}
    \label{eq:G-act}
(    \tilde \vt_\vr^+)^\sigma = \tilde \vt_{\vr^{\sigma}}^+
\end{equation}
for all $\sigma \in \GQ$.
Here $\vr^\sigma$ is the Hecke character defined in \eqref{eq:chisigc}.
\end{theorem}

The proof of this theorem follows from the explicit construction of harmonic Maass forms using regularized Petersson inner products, see Theorem~\ref{normalized harmonic Maass form} below. This method depends on Serre duality \cite{borcherdsduke} and has been applied in various contexts to produce real-analytic modular  forms \cite{ehlenduke, ES18, Via19, lischwagenscheidt}.
For CM forms of weight $k \ge 2$, such inner products can be evaluated in terms of CM values of Borcherds-Shimura lifts, which are meromorphic modular forms of weight $2k-2$.
Up to suitable periods, these values are algebraic numbers, see Proposition~\ref{prop:algGal}.
Then we need to suitably normalize the periods and understand the Galois action on these algebraic numbers, see Proposition~\ref{prop:Pnorm-compare}.
Putting these together, we arrive at Theorem \ref{thm:hmfchi}, which gives us vector-valued harmonic Maass forms having algebraic holomorphic part Fourier coefficients with compatible Galois actions.
To obtain Theorem \ref{thm:Gal} for scalar-valued newforms, we take suitable linear combinations of components of vector-valued harmonic Maass forms constructed in Theorem \ref{thm:hmfchi}. Here the crucial fact is that these newforms are themselves linear combinations of vector-valued binary theta series, see Proposition~\ref{prop:vv2sc}.

The following corollary of Theorem \ref{thm:Gal} affirms question 3(vi) above.
\begin{corollary}
  \label{cor:main}
  For a CM newform $\vt_\vr \in S_{k}^\new(N,\chi)$ defined over $F_\vr$, and any $\tilde\vt_\vr \in H_{2-k}(N,\overline{\chi})$ good for $\vt_\vr$, all coefficients of its holomorphic part $\tilde \vt_\vr^+$ are contained in $F_\vr$.
\end{corollary}

\begin{remark}
  In \cite{Can14}, Candelori gave an algebraic geometric interpretation of
  integral weight harmonic Maass forms. With this approach, he was able to prove that harmonic Maass forms associated with CM modular forms have algebraic holomorphic part Fourier coefficients.
  From private communication, he informed us that the same approach could give another proof of Corollary~\ref{cor:main}.
\end{remark}

\begin{remark}
In the case of weight $k = 1$, the harmonic Maass forms with CM newforms as $\xi_1$-images are studied in \cite{dukeli, ehlenduke, Via19, CL20}. The holomorphic part Fourier coefficients are shown to be $\Qbar$-linear combinations of logarithms of algebraic numbers. 
\end{remark}

It is generally believed that the Fourier coefficients of a generic harmonic Maass form are transcendental (see Remark 8 in \cite{bruinieronoheegnerdivisors}), and expected that the two known cases above are the only harmonic Maass forms with algebraic Fourier coefficients. 
In the unary case, the second and third authors showed that the Fourier coefficients are rational with explicitly bounded denominators \cite{lischwagenscheidt}. In the binary case, however, it seems that the algebraic coefficients have unbounded denominators. We formulate these observations into the following conjecture.

\begin{conjecture}
  \label{conj:1}
  Let $\tilde f$ be a harmonic Maass form of weight $k \in \frac{1}{2}\Zb$ for a congruence subgroup such that $\xi(\tilde f) \neq 0$. If the Fourier coefficients of its holomorphic part are all algebraic, then $\xi(\tilde f)$ is either a linear combination of unary theta series with $k \in \{\frac{1}{2}, \frac{3}{2}\}$ or of binary theta series with $2-k \ge 2$.
  In those cases, the coefficients have bounded denominators if and only if $k$ is half-integral. 
\end{conjecture}


The paper is organized as follows. Section~\ref{sec:prelim} contains preliminary information, and a relation between CM newforms and vector-valued binary theta series in Proposition~\ref{prop:vv2sc}.
In Section~\ref{subsec:algPet}, we prove the algebraicity of regularized Petersson inner product between weakly holomorphic modular forms and binary theta series.
Then in Section~\ref{subsec:Pnorm}, we compare the Petersson norm of binary theta series associated to different quadratic spaces in order to normalize the period in Proposition~ \ref{prop:Pnorm-compare}.
Finally, we prove the main result in Section~\ref{subsec:preimage}, and give a numerical example in the last section.\\

{\bf Acknowledgement}:
We thank Luca Candelori for private communication concerning \cite{Can14}, and sharing the slides of his talk at the AMS meeting at Portland in 2018.
Y. Li was supported by LOEWE research unit USAG, and by the Deutsche Forschungsgemeinschaft (DFG) through the Collaborative Research Centre TRR 326 ``Geometry and Arithmetic of Uniformized Structures", project number 444845124. M. Schwagenscheidt was supported by SNF projects 200021\_185014 and PZ00P2\_202210
\section{Preliminaries}
\label{sec:prelim}
Fix an algebraic closure $\Qbar \subset \Cb$ of $\Qb$,
which induces $\Gal(\Cb/\Qb) \to \GQ$.
Denote $\c \in \GQ$ the image of complex conjugation.
For $\alpha \in \Qbar$ and $\sigma \in \GQ$, we denote
$$
\alpha^\sigma := \sigma^{-1}(\alpha)
$$
its Galois conjugate.\footnote{The superscript makes the action a right action. Therefore we have an inverse in the definition.}
For any formal Laurent series $g = \sum_n c_n q^n$ with $c_n \in \Qbar$, we also denote
$$
g^\sigma := \sum_n c_n^\sigma q^n
$$
its Galois conjugate.
For a number field $E$, we denote $\Ab_E$, resp.\ $\hat E$, its ring of adeles, resp.\ finite adeles, and write $\Ab := \Ab_\Qb$. 

\subsection{Modular forms for the Weil representation}

Let $V$ be a rational quadratic space of signature $(b^+,b^-)$ with quadratic form $Q \colon V \to \Q$ and corresponding bilinear form $(\cdot,\cdot)$.
Let $L \subset V$ be an even lattice of level $N$ and let $L'$ be the dual lattice of $L$. We also let $L^-$ (resp.\ $V^-$) denote the lattice $L$ (resp.\ vector space $V$) with quadratic form $-Q$. The discriminant group $L'/L$ is finite and isomorphic to $\hat{L}'/\hat{L}$, where $\hat{L}' := L' \otimes \hat{\Z}$. For $\mu \in L'/L$ we consider the characteristic functions 
\begin{equation}
  \label{eq:phimu}
  \phi_\mu := \cha(\hat L + \mu), \quad \mu \in L'/L,
\end{equation}
which form a basis for the space
\[
\Sc_L := \bigoplus_{\mu \in L'/L}\C \phi_\mu \subset \Sc(V(\hat\Qb))
\]
of Schwartz functions on $V(\hat\Qb)$ which are supported on $\hat L'$ and which are constant on cosets of $\hat{L} := L \otimes \hat{\Z}$. We let $\langle \cdot,\cdot \rangle$ be the bilinear pairing on $\Sc_L$ defined by $\langle \phi_\mu,\phi_\nu \rangle = \delta_{\mu,\nu}$.

Let $\rho_L$ be the Weil representation of the two-fold metaplectic cover $\Mp_2(\Z)$ of $\SL_2(\Z)$ on $\Sc_L$ (see \cite[Section~4]{borcherds}). It is self-adjoint with respect to $\langle \cdot, \cdot \rangle$.
If  $L$ has level $N$, then its matrix coefficients with respect to the standard basis $\{\phi_\mu: \mu \in L'/L\}$ are defined over $\Gal(\Qb(\zeta_N)/\Qb)$. 
Theorem~4.3 in \cite{mcgraw} tells us that for any $\gamma \in \SL_2(\Zb)$
\begin{equation}
  \label{eq:gammaa}
\rho_L(\gamma)^{\sigma_a} = \rho_L(\gamma^a),~
\gamma^a := \smat{1}{}{}{a} \gamma \smat{1}{}{}{a^{-1}} \in \SL_2(\Zb/N\Zb),
\end{equation}
with 
$\sigma_a \in \Gal(\Qb(\zeta_N)/\Qb)$ the element associated with $a \in (\Zb/N\Zb)^\times$ by the Artin map.
Note that $\Gamma(N)$ is in the kernel of $\rho_L$, and $\gamma \mapsto \gamma^a$ is an automorphism of $\Gamma(N)\backslash \SL_2(\Zb)$.
Therefore for any $\Gamma_0(N)$-invariant function $\phi \in \Sc_{L}$ valued in $\Qb$, the function
\begin{equation}
  \label{eq:tr}
  \tr^N_1 \phi :=      \sum_{\gamma \in \Gamma_0(N) \backslash \SL_2(\Zb)}
  \rho_{L}(\gamma)(\phi) \in \Sc_L
\end{equation}
also has value in $\Qb$.

For $k \in \frac{1}{2}\Z$, let $H_{k, L}$ be the space of harmonic (weak) Maass forms of weight $k$ valued in $\Sc_L$ for the representation $\rho_L$  (see \cite{bruinierfunke04}) for which there is a Laurent polynomial 
\[
P_f(\tau) = \sum_{\mu \in L'/L}\sum_{\substack{m \in \Q_{\leq 0} \\ m \gg -\infty}}c_f^+(m,\mu)q^m \phi_\mu,
\]
called the principal part of $f$, such that $f(\tau)-P_f(\tau) = O(e^{-\varepsilon})$ as $v=\Im(\tau) \to \infty$, for some $\varepsilon > 0$. It contains $S_{k,L} \subset M_{k,L} \subset M_{k,L}^!$,  the subspaces of cusp forms, holomorphic, and weakly holomorphic modular forms. The space $H_{k,L}$ can only be non-trivial if
\[
\kappa := k-\frac{b^--b^+}{2}
\]
is an integer. In this case, the components of $f(\tau) = \sum_{\mu \in L'/L}f_\mu(\tau)\phi_\mu \in H_{k, L}$ satisfy the symmetry
\begin{equation}
  \label{eq:sym}
f_{-\mu}(\tau) = (-1)^\kappa f_{\mu}(\tau)  .
\end{equation}

For $k \neq 1$ every harmonic Maass form $f \in H_{k,L}$ has a splitting $f = f^+ + f^-$ into a holomorphic and a non-holomorphic part with Fourier expansions of the shape
\begin{align*}
f^+(\tau) &= \sum_{\mu \in L'/L}\sum_{\substack{m \in \Q \\ m \gg -\infty}}c_f^+(m,\mu)q^m \phi_\mu, \\
f^-(\tau) &= \sum_{\mu \in L'/L}\sum_{\substack{m \in \Q \\ m < 0}}c_f^-(m,\mu)\Gamma(1-k,4\pi|m|v)q^m \phi_\mu,
\end{align*}
with coefficients $c_{f}^\pm(m,\mu) \in \C$ and the incomplete Gamma function $\Gamma(s,x) = \int_{x}^\infty e^{-t}t^{s-1}dt$. The $\xi$-operator
\[
\xi_{k}f = 2iv^k \overline{\frac{\partial}{\partial \overline{\tau}}f}
\]
defines a surjective map $\xi_k \colon H_{k,L} \to S_{2-k,L^-}$.

We now fix a convenient basis for $M_{k,L}^!$ whose elements all have rational coefficients (see the appendix of \cite{BEY21}).


\begin{lemma}\label{basis modular forms}
	Let $k \in \frac{1}{2}\Z$ with $k \geq 2$. Furthermore, let $d = \dim(M_{k,L})$.
	 \begin{enumerate}
	 	\item There is a basis $\{g_1,\dots,g_d\}$ of $M_{k,L}$ with only rational Fourier coefficients, and indices $(n_1,\beta_1),\dots,(n_d,\beta_d)$ with $n_j+Q(\beta_j) \in \Z$ such that 
	 	\[
	 	c_{g_j}(n_l,\beta_l) = \frac{1}{2}\delta_{j,l}.
              \]
		\item For each $m \in \Q_{>0}$ and $\mu \in L'/L$ with $m + Q(\mu) \in \Z$ there exists a unique weakly holomorphic modular form $f_{m,\mu} \in M_{k,L}^!$ with only rational Fourier coefficients, having a Fourier expansion of the form
	\[
	f_{m,\mu} = \frac{1}{2}q^{-m}(\phi_\mu + (-1)^{\kappa} \phi_{-\mu}) + O(1)
	\]
	and with $c_{f_{m,\mu}}(n_j,\pm\beta_j) = 0$ for all $j \in \{1,\dots,d\}$.
	\end{enumerate}
	For $m \in \Q_{\leq 0}$ and $\mu \in L'/L$ we put $f_{m,\mu} = (\pm 1)^{\kappa}g_{j}$ if $(m,\mu) = (n_j,\pm\beta_j)$, and $f_{m,\mu} = 0$ otherwise. Then the non-zero forms $f_{m,\mu}$ with $m \in \Q$ and $\mu \in (L'/L)/\{\pm 1\}$ such that $m + Q(\mu) \in \Z$ form a basis of $M_{k,L}^!$.
\end{lemma}


Using the above basis for $M_{k,L}^!$ we can explicitly construct $\xi$-preimages of cusp forms whose coefficients are given by regularized Petersson inner products
\begin{align}\label{petersson inner product}
\langle f,g \rangle_{\Pet} := 
\lim_{T \to \infty}\int_{\mathcal{F}_T}\left\langle f(\tau),\overline{g(\tau)}\right\rangle v^k \frac{du dv}{v^2},
\end{align}
where $\mathcal{F}_T = \{\tau = u+iv \in \H: |\tau| \geq 1, |u| \leq 1/2, v \leq T\}$ is a truncated fundamental domain for $\SL_2(\Z)\backslash \H$, and $f \in M_{k,L}^!, g \in S_{k,L}$.
Theorem 2.14 in \cite{ES18} gives us the following result.

\begin{proposition}\label{normalized harmonic Maass form}
	Let $k \in \frac{1}{2}\Z$ with $k \geq 2$ and let $g \in S_{k,L}$ be a cusp form of weight $k$ for $\rho_L$. Then the generating series
	\[
	\tilde{g}^+(\tau) = \sum_{\mu \in L'/L}\sum_{m \in \Q}\langle f_{m,\mu},g\rangle_{\Pet}\, q^m\phi_\mu
	\]
	is the holomorphic part of a harmonic Maass form $\tilde{g} \in H_{2-k,L^-}$ with $\xi_{2-k}(\tilde{g}) = g$. Here $f_{m,\mu} \in M_{k,L}^!$ are the modular forms constructed in Lemma~\ref{basis modular forms}. 
\end{proposition}

The above result also holds in the scalar-valued case. For the rest of this section, let $k \in \Z$.
For any Laurent series $f$ in $q^{1/N}$,
denote
$\prin( f)$ its principal part, which is a polynomial in $q^{-1/N}$.
For a harmonic Maass form $\tilde f \in H_{2-k}(\Gamma(N))$ and a cusp form $g \in S_{k}(\Gamma(N))$, we denote
\[
\prin(\tilde f \cdot g) := \prin(\tilde f^+ \cdot g).
\]
An application of Stokes' theorem gives us
\begin{equation}
  \label{eq:stokes}
\begin{split}
[\SL_2(\Zb): \Gamma(N)]^{-1}
\langle \xi_{2-k}(\tilde f), g \rangle_{\Pet}
&=
 \sum_{c \in \Gamma(N) \backslash \mathbb{P}^1(\Qb)}
                                                     \CT(\prin(\tilde f \mid_{2-k}\gamma_c) \cdot (g \mid_{k} \gamma_c))\\
  &=
\CT\lp
\lp \tr^N_1   (\tilde f \cdot g)\rp^+\rp.    
\end{split}  
\end{equation}
Here $\gamma_c \in \Gamma(N)\backslash \SL_2(\Zb)$ is any element satisfying $\gamma(c) \cdot \infty = c$, and
 the superscript $+$ indicates the holomorphic part.
In particular, the sum on the right is always 0 when $\tilde f$ is weakly holomorphic. Conversely, given any collection of principal parts $\{P_c: c \in \Gamma(N)\backslash \mathbb{P}^1(\Qb)\}$ such that
\[
 \sum_{c \in \Gamma(N) \backslash \mathbb{P}^1(\Qb)}
\CT(P_c \cdot (g \mid_{k} \gamma_c))
= 0
\]
for all $g \in S_k(\Gamma(N))$, then there exists $f \in M^!_{2-k}(\Gamma(N))$ with $\prin(f \mid_{2-k} \gamma_c) = P_c$. 

\subsection{CM cusp forms}
\label{subsec:setup}
Let $\k = \Qb(\sqrt{-D}) \subset \Qbar \subset \Cb$ be an imaginary quadratic field of fundamental discriminant $-D < 0$. Let $I_\mf$ be the fractional ideals of $\k$ coprime to a modulus $\mf$  and
\begin{equation}
  \label{eq:Pmmu}
P_{\mf, \mu}= \{\alpha \Oc_\k \in I_\mf: \alpha \equiv \mu \bmod{\mf}\}, \quad \mu \in (\Oc_\k/\mf)^\times.
\end{equation}
An algebraic Hecke character  $\vr$  of weight $k \ge 1$ is a multiplicative function on $I_\mf$ satisfying
\begin{equation}
  \label{eq:prin}
\vr(\alpha \Oc_\k) = \alpha^{k-1}  , \quad \alpha \Oc_\k \in P_{\mf, 1} \subset I_\mf.
\end{equation}
We say that it has  modulus $\mf$ if $\mf$ is the largest ideal such that \eqref{eq:prin} is satisfied.

The image of $\vr$ is contained in $\Qb(\vr)^\times$ for a number field $\Qb(\vr) \subset \Cb$, which contains $\k$.
Note that $\Qb(\vr)/\Qb$ is not necessarily Galois.
For any $\sigma \in \Gal(\Qbar/\k)$, the character
\begin{equation}
  \label{eq:chisig}
  \vr^\sigma(\af) := (\vr( \af))^\sigma
\end{equation}
is an algebraic Hecke character with the same modulus and infinity type as $\vr$.
On the other hand for complex conjugation $\c \in \GQ$, the character
\begin{equation}
  \label{eq:chic}
  \vr^\c(\af) := \overline{\vr(\overline \af)}
\end{equation}
is an algebraic Hecke character modulo $\overline \mf$ with the same infinity type as $\vr$.
So for any $\sigma \in \GQ$,
 we set
\begin{equation}
  \label{eq:chisigc}
  \vr^\sigma(\af) :=
  (\vr(\sigma(\af))^\sigma,  
\end{equation}
which is an algebraic Hecke character of modulus $\sigma(\mf)$ and weight $k$.
This agrees with \eqref{eq:chisig} when $\sigma \in \Gal(\Qbar/\k)$ since $\sigma(\af) = \af$ in that case.

The generating series
\begin{equation}
  \label{eq:vartheta}
  \vartheta_\vr(\tau) := \sum_{\af \subset \Oc} \vr(\af) q^{\Nm(\af)}
\end{equation}
is a CM newform in $S_{k}(N, \chi_\k \eta)$ with $N := DM, M:= \Nm(\mf)$ and 
\begin{equation}
  \label{eq:chars}
\chi_\k(d) := \lp \frac{-D}{d} \rp, \quad
\eta(\alpha) := \frac{\vr(\alpha \Oc)}{\alpha^{k-1}}, \quad \alpha \in \k^\times.
\end{equation}
Here $\chi_\k \eta$ is a Dirichlet character modulo $N$ of parity $(-1)^{k}$. 
For any $\sigma \in \GQ$, we have 
\begin{equation}
  \label{eq:chi-chic}
  ( \vartheta_\vr)^\sigma =  \vartheta_{\vr^{\sigma}}.
\end{equation}
Let $F_\vr \subset \Qb(\vr)$ denote the field generated by the Fourier coefficients of $\vartheta_\vr$. 
It is well-known  \cite[Proposition~3.2]{Ribet77} that $F_\vr$ is Kroneckerian, i.e.\ either totally real or a CM field, and contains the image of $\eta$.
%
We record here a result about the field $\Qb(\vr)$, which gives another proof of this fact.
\begin{lemma}
  \label{lemma:field}
  In the notations above, the extension $\Qb(\vr)/F_\vr$ is Galois with Galois group an elementary 2-group.
  Furthermore, $\Qb(\vr)$ is Kroneckerian.
\end{lemma}
\begin{proof}
  By the Chebotarev Density Theorem, $\Qb(\vr)$ is generated by $\vr(\pf)$ over primes $\pf \nmid \mf\df$ in $\k$. Denote $c_n$ the $n$-th Fourier coefficient of $\vartheta_\vr$.
  If $p$ is inert, then $\vr(p \Oc_\k) = c_{p^2} \in F_\vr$.
  If $p = \pf \bar\pf$ is split, then $\vr(\pf) + \vr(\bar\pf) = c_p$ and $\vr(\pf) \vr(\bar\pf) = c_p^2 - c_{p^2}$ are in $F_\vr$, which means $F_{\vr}(\vr(\pf))/F_\vr$ has degree at most 2.
  Suppose $\Qb(\vr) = F_\vr(\vr(\pf_1), \dots, \vr(\pf_r))$, then $\Qb(\vr)/F_\vr$ is the composite of the quadratic extensions $F_\vr(\vr(\pf_i))/F_\vr$, hence Galois with elementary 2-group as its Galois group.

  To prove the second claim,  let $\pf$ be an unramified prime in $\k$.
  We will show that
  \begin{equation}
    \label{eq:abs}
   |\vr(\pf)^\sigma|^2 = \Nm(\pf)^{k-1}
  \end{equation}
  for all  $\sigma \in \GQ$. Then Theorem 2 in \cite{BL78} (see also section 4(a) loc.\ cit.) implies that $\Qb(\vr(\pf))$ is Kroneckerian.  Therefore as the composite of Kroneckerian fields, $\Qb(\vr)$ is also Kroneckerian. This gives another proof that $F_\vr \subset \Qb(\vr)$ is Kroneckerian.   

  To prove \eqref{eq:abs}, suppose   $p = \pf \overline\pf$ is a split prime.
  Since $\vr(\pf)^N = \vr(\pf^N) = \lambda^{k-1}$ for some $N \in \Nb$ and $\lambda \in \k$ satisfying $\lambda \Oc_\k = \pf^N$, we have
  $$
  |\vr(\pf)^\sigma|^{2N} = (\vr(\pf)^\sigma)^{N}( \overline{\vr(\pf)^\sigma})^{N}
  = (\vr(\pf)^N)^\sigma \overline{(\vr(\pf)^N)^\sigma} = \Nm(\lambda)^{k-1} = p^{(k-1)N}.
  $$
  When $\pf = p \Oc_\k$ is inert, the claim \eqref{eq:abs} is clear since
  $
\vr(\pf) = p^{k-1} \eta(p)
$
with $\eta(p)$ a root of unity.
\end{proof}

\begin{remark}
  \label{rmk:commutec}
  Since $\Qb(\vr)$ is Kroneckerian, so is its Galois closure, and the complex multiplication is in the center of this Galois group.
  This in particular implies
  \begin{equation}
    \label{eq:commutec}
    \overline{\vr(\af)}^\sigma =     \overline{\vr(\af)^\sigma}
  \end{equation}
  for all fractional ideals $\af$ of $\k$ and $\sigma \in \GQ$.
\end{remark}
\begin{example}
\label{ex:eta8}  
Consider the newform $\eta(3\tau)^8 \in S_{4}(9)$, which has  CM by $\k = \Qb(\sqrt{-3}) = \Qb(\zeta)$ with $\zeta = \zeta_6 = e^{2\pi i /6}$.
Then by taking $\mf = \df = \sqrt{-3} \Oc_\k$, we have  $ I_\mf = P_{\mf, 1}$ and
$\vartheta_\vr(\tau) = \eta(3 \tau)^8$ with $\vr$ the Hecke character uniquely determined by
\begin{equation}
  \label{eq:vteta8}
  \vr(\alpha\Oc_\k) = \alpha^3, \quad \alpha  \in \df + 1,
\end{equation}
by \eqref{eq:prin}, and $\Qb(\vr) = F_\vr = \Qb$. 
Furthermore, we can relate $\eta(\tau)^8$ to components of vector-valued binary theta series.
Consider the lattice $(P, Q) = (\Oc_\k, \Nm)$, whose dual lattice is $P' = \df^{-1}$ with $\df  = \sqrt{-3}\Oc_\k$, and $P'/P = \df^{-1}/\Oc_k \cong \Zb/3\Zb$ via $a + b\zeta \mapsto a$. 
Then the cusp form
\begin{equation}
  \label{eq:eta8}
  \begin{split}
    \theta_{P, 4}(\tau)
    &= \sum_{\mu \in P'/P} \phi_\mu \sum_{\lambda \in P + \mu} \lambda^3 q^{Q(\lambda)}\\
&= \sum_{r \in (\Zb/3\Zb)^\times} \phi_{r} \sum_{\substack{a, b \in \Zb + \frac{r}{3}}} (a + b\zeta)^3 q^{a^2 + ab + b^2}
= \frac{\sqrt{-3}}{3} \eta(\tau)^8  (\phi_{1} - \phi_{-1})
  \end{split}
\end{equation}
generates the 1-dimensional vector space $S_{4, P}$.
\end{example}

\subsection{Binary theta series}
Now, we will write the CM form $\vartheta_\vr$ as a linear combination of binary theta functions (see \cite[Section~6]{CL20} for the case of RM forms of weight one).
It is in some sense similar to \eqref{eq:eta8}, though we will not scale the variable $\tau$.
For later purposes, we will work in the adelic setting.

For the quadratic space
\begin{equation}
  \label{eq:Ua}
(U, Q) = (U_a, Q_a) = (\k, a \Nm)  
\end{equation}
with $a \in \Qb_{>0}$, we identify $T := \GSpin(U) \cong\mathrm{Res}_{\k/\Qb} \Gm$, and
\begin{equation}
  \label{eq:SOU}
  \SO_U(R) \cong R^1 := \k^1 \otimes R, \quad \k^1 := \{z \in \k: \Nm(z) = 1\},
\end{equation}
for any $\Qb$-algebra $R$.
Under these isomorphisms, the surjective map $T \to \SO_U$ is given by
$$
\Nm^-\colon T \to \SO_U, \quad
z \mapsto  z/\bar{z},
$$
with $\bar{z}$ the Galois conjugation of $z$.
 Let $\phi^k = \phi \phi^k_\infty \in \Sc(U(\Ab)) \cong \Sc(\Ab_\k)$ be a Schwartz function such that $\phi \in \Sc(\hat \k)$ and
$$
\phi^k_\infty\colon \k \otimes \Rb \cong \Cb \to \Cb, \quad  z\mapsto  z^{k-1} e^{-2\pi z \overline z}.
$$
Then we define the binary theta series
\begin{equation}
  \label{eq:thetaUk}
      \theta_{U, k}(\tau, \phi)
      := a^{(1-k)/2}v^{-k}
\sum_{x \in U} (\omega(g_\tau) \phi^k)(x),
\end{equation}
where $\omega = \omega_a$ denotes the Weil representation of $\SL_2(\A) \times \mathrm{O}(U)(\A)$ on $\Sc(U(\A))$ (see \cite{Weil2}), and $g_\tau = \left( \begin{smallmatrix}\sqrt{v} & u/\sqrt{v} \\ 0 & 1/\sqrt{v} \end{smallmatrix}\right) \in \SL_2(\Rb) \subset \SL_2(\A)$ satisfies $g_\tau \cdot i = \tau$.
The binary theta series is in $S_{k}(\Gamma_K)$ with $\Gamma_K := K \cap \SL_2(\Qb) \subset \SL_2(\Zb)$ a congruence subgroup and 
$K \subset \SL_2(\hat\Zb)$ an open compact subgroup fixing $\phi$.


Let $\vr^{-1} \vr_\infty^k: \k^\times\backslash \Ab_\k^\times \to \Cb^\times$ be a continuous homomorphism such that
$$
\vr^k_\infty\colon (\k \otimes \Rb)^\times \cong \Cb^\times \to \Cb^\times, \quad  z \mapsto z^{k-1}. 
$$
and  $\vr \colon \hat \k^\times \to \Qbar^\times \hookrightarrow \Cb^\times$ is a continuous function.
Suppose that $\mf$ is the largest ideal such that $\vr$ is trivial on
$$
  K(\mf) :=
  \prod_{v \nmid \mf} \Oc_{\k_v}^\times
  \times  \prod_{v \mid \mf} ( 1 + \mf \Oc_{\k_v})
\subset   \hat \k^\times.
$$
Using the isomorphism
\begin{equation}
  \label{eq:Imiso}
  \begin{split}
      I_\mf &\cong \sideset{}{'}\prod_{v < \infty,~ v \nmid \mf}(\k_v^\times/\Oc_{\k_v}^\times)
=  \hat \k_{(\mf)}^\times / (  \hat \k_{(\mf)}^\times  \cap K(\mf)),\\
  \end{split}
\end{equation}
where $\hat \k_{(\mf)}^\times
:=   \sideset{}{'}\prod^{}_{v < \infty,~ v \nmid \mf}\k_v^\times \times \prod_{v \mid \mf} (1 + \mf \Oc_{\k_v})$, 
 we can view $\vr$ as a function on $I_\mf$, and a Hecke character of weight $k$ and modulus $\mf$.
Conversely, any Hecke character $\vr$ is the finite part of such a continuous homomorphism.

For an open compact subgroup $K_\f \subset \hat{\Oc}_\k^\times$, we define 
\begin{equation}
  \label{eq:ClK}
  \Cl(K_\f) :=
\k^\times \backslash \hat \k^\times /K_\f
\cong   T(\Qb)\backslash T(\hat \Qb) / K_\f,
 \end{equation}
 the (finite) class group for $K_\f$.
 The image of $\Nm^-\colon  T \to \SO_U \subset \mathrm{Res}_{\k/\Qb} \Gm$ gives rise to the subgroup 
 \begin{equation}
   \label{eq:Cl-}
   \Cl^1(K_\f)
   :=\Nm^-(\Cl(K_\f))    =  \k^1 \backslash \hat \k^1 /K^1_\f \subset \Cl(K_\f)
 \end{equation}
 with $K^1_\f:= \Nm^-(K_\f) = K_\f \cap \hat \k^1$ an open compact in $\hat \k^1$. 
 Suppose $K_\f^1 \subset \ker(\vr)$, then we have
\begin{equation}
  \label{eq:thetachi}
  \begin{split}
    \theta_\vr(\tau, \phi)
    &:= a^{(1-k)/2}v^{-k}\int_{T(\Qb)\backslash T(\hat \Qb)} \vr^{-1}\circ \Nm^-(h^{}) \theta_{U, k}(\tau, \omega_a(h) \phi) dh    \\
    &= \vol(K_\f) \sum_{h \in \Cl(K_\f)}
    \vr(h/\bar h)^{-1} \theta_{U, k}(\tau, \omega_a(h) \phi)\\
    &= \vol(K_\f) \sum_{h \in \Cl(K_\f)}
    \vr(h/\bar h)^{-1} \sum_{x \in \k} \phi((\bar h/h)  x) x^{k-1} q^{a \Nm(x)}
  \end{split}
\end{equation}
where we normalize the Haar measure such that
\begin{equation}
  \label{eq:vol}
  \vol(K_\f) = [\Cl(\hat{\Oc}_\k^\times): \Cl(K_\f)]^{-1} = [\hat{\Oc}_\k^\times : K_\f]^{-1}
  \cdot |\Oc_\k^\times|  \cdot |K_\f \cap \Oc_\k^\times|^{-1}.
\end{equation}

Suppose we take $\phi = \phi_{\mu}$ as in \eqref{eq:phimu} for $P \subset U(\Qb)$ a lattice and $\mu \in P'/P$, then for $h \in \SO(U)(\Ab)$
\begin{equation}
  \label{eq:thetaP}
  \Theta_{P}(\tau, h)
  :=  \sum_{\mu \in L'/L} \theta_{U, k}(\tau, \omega(h) \phi_\mu)\phi_\mu
\end{equation}
is a vector-valued cusp form in $S_{k, P}$.
The same holds for
\begin{equation}
  \label{eq:thetachiP}
  \begin{split}
    \Theta_{ \vr, P}(\tau)
    &=  \sum_{\mu \in P'/P}\theta_{\vr, P+\mu}(\tau)\phi_\mu,\\
    \theta_{\vr, P + \mu}(\tau) &:= \theta_{\vr}(\tau  , \phi_\mu)
    = \vol(K_\f) \sum_{h \in \Cl(K_\f)} \vr(\bar h/h)  \Theta_{P}(\tau, h, \phi_\mu).
  \end{split}
\end{equation}
Note that $\Theta_{P}(\tau, h)$ and $\Theta_{h \cdot P}(\tau, 1)$ have the same components, which could be permuted.

For a modulus $\mf \subset \Oc_\k$, denote $\Cl(\mf) := \Cl(K(\mf)) := \k^\times \backslash \Ab_\k^\times / K(\mf)$.
Then \eqref{eq:Imiso} gives us
$\Cl(\mf) \cong I_\mf/P_{\mf, 1}$.
Fix $t_1, \dots, t_\ell, h_1, \dots, h_m \in \hat \k^\times_{(\mf)}$ such that
\begin{equation}
  \label{eq:tihj}
\{t_i \Nm^-(h_j): 1 \le i \le \ell,~ 1 \le j \le m\}  
\end{equation}
is a set of representatives of $\Cl(\mf)$.
Let $\af_i, \bfrak_j \in I_\mf$ be the fractional ideals corresponding to $t_i$ and $h_j$ respectively under \eqref{eq:Imiso}.
Then  $a_i := \Nm(\af_i)^{} \in \Qb_{> 0}$ satisfies $ \Nm( t_i)/a_i \in \hat\Zb^\times$ for all $1 \le i \le \ell$.
For any fixed $\mu \in (\Oc_\k/\mf)^\times$, the map
\begin{equation}
  \label{eq:map1}
  \begin{split}
  \left\{(\af_i, \bfrak_j, \cf): 1 \le i \le \ell, 1 \le j \le m, \cf \in P_{\mf, \mu}, \cf \subset \af_i \Nm^-(\bfrak_j)
\right\}
&\to I_\mf\\
(\af_i, \bfrak_j, \cf) &\mapsto  \cf/(\af_{i}\Nm^-(\bfrak_{j}))
\end{split}
\end{equation}
is an injection
with image being the subset of integral ideals in $I_\mf$.
%

Consider the following even integral lattice in $U_{1/a_i}$.
\begin{equation}
  \label{eq:Pj}
  (P_i, Q_i) := \lp \mf \af_i, \frac{1}{a_i} \Nm\rp,
\end{equation}
which has level $N = DM$ with dual lattice $P_i' = \overline\mf^{-1} \df^{-1}\af_i$.
We can express the CM newform $\vartheta_\vr$ from \eqref{eq:vartheta} in terms of the components of binary theta series attached to these lattices.
\begin{proposition}
  \label{prop:vv2sc}
Let $\vr$ be a Hecke character of modulus $\mf$ and weight $k$. 
Then in the notations above, we have
\begin{equation}
  \label{eq:vv2sc2}
  \begin{split}
    \vartheta_\vr(\tau)
    &=
      \frac{1}{ |\Oc_\k^\times|}
  \sum_{\substack{1 \le i  \le \ell\\ \mu \in (\Oc_\k/\mf)^\times}}
  \vr(t_{i })^{-1}
  \eta(\mu) \theta_{\vr, P_{i } + \mu}(\tau).
  \end{split}
\end{equation}

\end{proposition}
\begin{proof}
  For $h_j$ as in \eqref{eq:tihj}, we denote
\begin{equation}
  \label{eq:Pij}
  P_{i, j} := h_j \cdot P_i = \mf \af_i \Nm^-(\bfrak_j).
\end{equation}
The inclusion $P_{i, j} \subset \k$ identifies
\begin{equation}
  \label{eq:Pijid}
P_{i, j}'/P_{i, j} \cong   (\overline\mf \df)^{-1}/\mf.
\end{equation}
For fixed $i$, this identification gives an isometry between the finite quadratic modules $P_{i, j}'/P_{i, j}$ and $P_{i, j'}'/P_{i, j'}$ for $1 \le j, j' \le m$.
Furthermore, it preserves the subsets
\begin{equation}
  \label{eq:Pjisom}
(  \Oc_\k /\mf)^\times \subset  \Oc_\k /\mf.
\end{equation}
Since $h_j \in \hat \k^\times_{(\mf)}$, we have
\begin{equation}
  \label{eq:coset2}
  h_j \cdot (P_i + \mu) = P_{i, j} + \mu
\end{equation}
 for all $1 \le i \le \ell, 1 \le j \le m$ and $\mu \in \Oc_\k/\mf$.

We can now apply the map \eqref{eq:map1} and write 
\begin{align*}
  \vartheta_\vr(\tau)
  &=
    \sum_{\substack{1 \le i  \le \ell\\ 1 \le j  \le m}} {\vr}^{-1}( \af_{i } \Nm^-(\bfrak_{j }))^{}
  \sum_{\cf \subset \af_{i } \Nm^-(\bfrak_{j }),~ \cf \in P_{\mf, 1}}
  \vr(\cf)
  q^{ \Nm(\lambda)/ a_{i }}\\
  &= 
    \frac{\eta(\epsilon)}{ |\Oc_\k^\times \cap K(\mf)|}
    \sum_{\substack{1 \le i  \le \ell\\ 1 \le j  \le m}} \vr(t_{i })^{-1} {\vr}(\Nm^-(h_{j }))^{-1}
    \sum_{x \in  P_{i, j } + \epsilon}  x^{k-1} q^{ Q_{i }(x)}
\end{align*}
for any $\epsilon \in \Oc_\k^\times$.

More generally for any $\mu \in (\Oc_\k/\mf)^\times$, we have
\begin{equation}
  \label{eq:coset}
  \vartheta_\vr(\tau)
  =
    \frac{\eta(\mu)}{ |\Oc_\k^\times \cap K(\mf)|}
    \sum_{\substack{1 \le i  \le \ell\\ 1 \le j  \le m}} \vr(t_{i })^{-1} {\vr}(\Nm^-(h_{j }))^{-1}
    \sum_{x \in  h_j \cdot (P_{i } + \mu)}  x^{k-1} q^{ Q_{i }(x)}
  \end{equation}
by \eqref{eq:coset2}. Averaging over $\mu \in (\Oc_\k/\mf)^\times$, we have
\begin{equation*}
  \vartheta_\vr(\tau) =
  \frac{1}{ |\Oc_\k^\times \cap K(\mf)| \cdot  | (\Oc_\k/\mf)^\times| \cdot \vol(K(\mf)) }
  \sum_{1 \le i  \le \ell} \vr(t_{i })^{-1}
\sum_{\mu \in (\Oc_\k/\mf)^\times}\eta(\mu) \theta_{\vr, P_{i } + \mu}(\tau).
\end{equation*}
Since $  | (\Oc_\k/\mf)^\times|  = [\hat{\Oc}_\k^\times: K(\mf)]$, we can apply \eqref{eq:vol} to write
$|\Oc_\k^\times \cap K(\mf)| \cdot \vol(K(\mf)) \cdot |\Oc_\mf^\times| = |\Oc_\k^\times|$, giving us \eqref{eq:vv2sc2}.
\end{proof}

\begin{example}
  \label{ex:eta8a}
  Let $\vr$ be the Hecke character in Example \ref{ex:eta8}.
  Then $\Cl(\mf) = \Cl(\df)$ is trivial and the set in \eqref{eq:tihj} has size 1.
  So we have a single lattice $P = (\df, \Nm)$, whose dual lattice is $\df^{-2}$. 
  Proposition~\ref{prop:vv2sc} tells us that
  \begin{align*}
    \eta(3\tau)^8
    &= \vartheta_\vr(\tau) = (\theta_{\vr, P + 1}(\tau) - \theta_{\vr, P - 1}(\tau))/6
    =  \lp\sum_{x \in \df + 1} x^3 q^{\Nm(x)} -  \sum_{x \in \df - 1} x^3 q^{\Nm(x)}\rp/6,
  \end{align*}
  where we have used  $    \vol(K(\mf)) = 1$ and $|\Oc_\k^\times| = 6$.
\end{example}
\subsection{Quadratic spaces and theta functions}
Let $(V, Q)$ be a rational quadratic space of signature $(1, 2)$, and identify  $$
V(\Rb) \cong M_2(\R)^0 = \left\{\begin{pmatrix}x_3/2 & x_1 \\ x_2 & -x_3/2\end{pmatrix} \in M_2(\Rb): x_i \in \Rb\right\}
$$
with determinant as the quadratic form.
Then $H(\Rb) \cong \mathrm{GL}_2(\Rb)$ via the conjugation action, where $H := \GSpin(V)$.
  For $z = x +iy \in \H$, define
  $$
h_z := n(x)m(\sqrt{y}) \in \SL_2(\Rb) \subset \GL_2(\Rb) \cong H(\Rb)
$$
with $n(b) := \smat{1}{b}{}{1}, m(a) := \smat{a}{}{}{a^{-1}}$, which  satisfies  $h_z \cdot i = z$. 
Let $\Db$ be the Grassmannian of oriented negative two-planes in $V(\Rb)$.
It has two connected components $\Db^\pm$ with $\Db^+$ containing $Z_0 :=  \Rb \smat{}{1}{1}{} \oplus \Rb \smat{ 1}{}{}{- 1}$.
We can identify $\H$ with $\Db^+$ via
\begin{align}\label{identification DH}
  z &\mapsto h_z \cdot Z_0
=       \R \Re \begin{pmatrix}z & -z^2 \\ 1 & -z \end{pmatrix} \oplus \R \Im \begin{pmatrix}z & -z^2 \\ 1 & -z \end{pmatrix}.
\end{align}
Note that for any $\gamma = \smat{*}{*}{c}{d} \in \SL_2(\Rb)$, we have 
\begin{equation}
  \label{eq:decompose}
  h_{\gamma \cdot z} = \gamma h_z \kappa(\theta(\gamma, z)), \quad \kappa(\theta) := \smat{\cos\theta}{-\sin\theta}{\cos\theta}{\sin\theta}, \quad
  e^{i \theta(\gamma, z)} = |cz + d|/(cz+d).
\end{equation}
For any compact open subgroup $K \subset H(\A_f)$, the associated Shimura variety
\begin{equation}
  \label{eq:XK}
X_K := H(\Qb) \backslash \Db \times H(\A_f)/K
\end{equation}
is a disjoint union of modular (resp.\ Shimura) curves when $V$ is isotropic (resp.\ anisotropic).

\begin{example}
  \label{ex:XN}
  Let $N$ be a positive integer and consider the rational quadratic space
  \begin{equation}
    \label{eq:VN}
  (V, Q) = ( M_2(\Q)^0, N\det(x))
  \end{equation}
  of  signature $(1,2)$. 
  The corresponding bilinear form is given by $(x,y) = -N\tr(xy)$.
  For each prime $p$ let
\[
K_p(N) = \left\{\begin{pmatrix}a & b \\ c & d \end{pmatrix} \in \GL_2(\Z_p): c \in N \Z_p\right\}
\]
and let $K(N) := \prod_p K_p(N)$. Then $K(N) \subset H(\A_f)$ is a compact open subgroup and satisfies  $H(\A_f) = H(\Q)K(N)$ by strong approximation (see Theorem 3.3.1 in \cite{bump}).
Furthermore, $K(N) \cap H(\Qb) = \Gamma_0(N) \times \{\smat{-1}{}{}{1}\}$ and
$$
\Gamma_0(N) \backslash \Hb \cong X_{K(N)}, \quad z \mapsto [z, 1].
$$
\end{example}

Let $\omega = \otimes \omega_p$ be the Weil representation of $(G\times H)(\Ab)$ on $\Sc(V(\Ab))$, and define
  $$
  \theta_{}(\tau, (z, h), \varphi) = 
  \sum_{x \in V(\Q)}
(\omega(g_\tau, h_z) \varphi)(h^{-1}x)
$$
for $\varphi \in \Sc(V(\Ab))$. 
When $\varphi = \varphi_k = \varphi_f \varphi_{\infty,k} \in \Sc(V(\A))$ with $k \in \Nb$ and
\begin{equation}
  \label{eq:varphik}
  \varphi_{\infty,k}(x_1, x_2, x_3) := (x_2 - ix_3)^k e^{-\pi (x_1^2  + x_2^2 + x_3^2 )} \in \Sc(V(\Rb)),
\end{equation}
we denote
\begin{equation}
  \label{eq:thetak}
  \theta_{k}(\tau, (z, h), \varphi_f) := \sqrt{v}^{2k + \frac{1}{2}}\theta(\tau, (z, h), \varphi_{k})
\end{equation}
which has weight $2k$ in $z$.

For an even, integral lattice $L \subset V$, 
let $K(L) \subset H(\A_f)$ be the largest open compact subgroup preserving $\hat L$ and acting trivially on $\hat L'/\hat L \cong L'/L$.
If $K$ is contained in $K(L)$, then we denote
\begin{equation}
  \label{eq:thetaLk}
  \begin{split}
    \Theta_{L,k}(\tau, z, h) &:= \sum_{\mu \in L'/L} \theta_k (\tau, (z, h), \phi_\mu)  \phi_\mu\\
&=
    v^{k+1} y^{-2k}
    \sum_{\mu \in L'/L}
    \sum_{\lambda \in h(L+\mu)}
\left(\lambda,\begin{pmatrix}-\overline z & \overline z^2 \\ -1 & \overline z \end{pmatrix} \right)^{k}\ebf(Q(\lambda_{z^\perp})\tau + Q(\lambda_{z})\overline{\tau})\phi_{\mu}
  \end{split}
\end{equation}
 the vector-valued theta kernel with $\phi_\mu$ as in \eqref{eq:phimu} with $[z, h] \in X_K$. Here $\lambda_z$ denotes the orthogonal projection of $\lambda$ to the negative plane corresponding $z$, and $\lambda_{z^\perp}$ denotes the projection of $\lambda$ to the orthogonal complement in $V(\R)$ of the plane corresponding to $z$.
By \cite[Theorem~4.1]{borcherds} the theta function $\Theta_{L,k}(\tau,z,h)$ has weight $-k-\frac{1}{2}$ in $\tau$ for $\rho_L$.
We let $f \in M_{k+\frac{1}{2},L}^!$ be a weakly holomorphic modular form of weight $k+\frac{1}{2}$ for $\rho_L$. Its Borcherds-Shimura theta lift is defined as
\[
  \Phi_{L,k}(f,z,h) = \left\langle f(\tau),v^{-k-\frac{1}{2}}\overline{\Theta_{L,k}(\tau,z,h)}\right\rangle_{\Pet}
\]
with the regularized Petersson inner product as defined in \eqref{petersson inner product}.
Using the Fourier expansion computed by Borcherds in  Theorem 14.3 of \cite{borcherds}, we have the following result.

\begin{theorem}\label{shimura lift}
  The Borcherds-Shimura theta lift $\Phi_{L,k}(f,z,h)$ is a meromorphic modular form of weight $2k$ on $X_K$ and has the same field of definition as $f$.
Furthermore for any $\sigma \in \Aut(\Cb/\Qb)$, we have $\Phi_{L, k}(f, \cdot)^\sigma = \Phi_{L, k}(f^\sigma, \cdot)$. 
\end{theorem}

\begin{proof}
  Since $M^!_{k+\frac{1}{2}, L}$ has a basis of forms with rational Fourier coefficients \cite{mcgraw}, we can suppose that $f$ has rational Fourier coefficients.
  When $V$ is isotropic, Borcherds calculated the Fourier expansion of $\Phi_{L, k}$ in Theorem 14.3 of \cite{borcherds}. 
  The proof of the claim is then analogous to the case of Borcherds products, i.e.\ $k = 0$, studied by Howard and Madapusi-Pera \cite{HMP20},  and follows from the $q$-expansion principle (see Propositions 4.6.3 and 5.4.2 loc.\ cit.).
  In fact, the case of $k \ge 1$ is easier since there is no need to normalize $\Phi_{L, k}$ and its Fourier expansions by constants.
  When $V$ is anisotropic, one can apply the algebraic embedding trick developed in \cite{HMP20} to conclude the result (see sections 6 and 9 loc.\ cit.\ for details). 
\end{proof}

For a negative definite subspace $U^- \subset V$, we have two points $\{z_U^\pm \} \subset \Db$ given by $U^-(\Rb) \subset V(\Rb)$ with two possible orientations.
The group $T = \GSpin(U^-) = \GSpin(U)$ is isomorphic to $\Res_{\k/\Qb} \Gm$ for an imaginary quadratic field $\k$, and embeds into $H$, which gives us a CM cycle
\begin{equation}
  \label{eq:ZU}
Z(U) = T(\Qb) \backslash (\{z_U^\pm \times T(\hat \Qb)/(K \cap T(\hat \Qb))) \to X_K.   
\end{equation}
  For any $h \in \k^\times \backslash \hat \k^\times$, let $\sigma_h \in \Gal(\k^{\mathrm{ab}}/\k)$ be the image under the Artin map.
  Let $\Psi$ be  a meromorphic function on $X_K$ defined over $\Qbar$ and $\mathrm{Div}( \Psi)\cap Z(U) = \emptyset$,
  By the theory of complex multiplication (see \cite[Theorem 6.31]{shimauto}),   we have $\Psi(z_0, h') \in \Qbar$ for all   $h' \in \hat \k^\times$.
  Furthermore,
\begin{equation}
  \label{eq:rec}
  (\Psi(z_0, h'))^{  \sigma_{h}} =   \Psi^\sigma(z_0, h^{-1}h'),~
  (\Psi(z_0, h'))^\c =   \overline{\Psi}(z_0, \overline{h'})
\end{equation}
for all
 $\sigma \in \Gal(\Qbar/\k)$ whose image under $\Gal(\Qbar/\k) \to \Gal(\k^{\mathrm{ab}}/\k)$ is $\sigma_h$. 
This also generalizes to the case when $\Psi$ is a ratio of nearly holomorphic modular forms \cite{shimura75}.

\section{Main Results}
\subsection{Algebraicity of regularized inner products.}
\label{subsec:algPet}
In this section, we will prove a result about algebraicity and Galois action on ratios of certain regularized Petersson inner products involving binary theta series.
This result plays an important role in the proof of Theorem~\ref{thm:Gal} since these inner products will appear as the Fourier coefficients of the holomorphic part of the harmonic Maass forms in Theorem~\ref{thm:hmfchi} below.

\begin{proposition}
  \label{prop:algGal}
  Let $\vr$ be a Hecke character of weight $k$ for an imaginary quadratic field $\k$.
  For a lattice $P$ in a quadratic space $U$ as in \eqref{eq:Ua}, let $\Theta_{\vr, P} \in S_{k, P}$ be the vector-valued theta series in \eqref{eq:thetachiP}.
  For any $f \in M^!_{k, P}$ defined over $\Qbar$, the ratio $        \frac{  \langle f(\tau), \Theta_{ \vr, P}(\tau) \rangle_{\Pet}}{  \| \Theta_{\vr, P}(\tau) \|^2_{\Pet}}$ is in $\Qbar$.
  Furthermore, 
  \begin{equation}
  \label{eq:Petratio2}
  \begin{split}
        \lp \frac{  \langle f(\tau), \Theta_{ \vr, P}(\tau) \rangle_{\Pet}}{  \| \Theta_{\vr, P}(\tau) \|^2_{\Pet}} \rp^{    \sigma_{}}
&=      \frac{  \langle f^\sigma(\tau), \Theta_{ \vr^\sigma, P}(\tau) \rangle_{\Pet}}{  \| \Theta_{\vr^\sigma, P}(\tau) \|^2_{\Pet}} 
  \end{split}
\end{equation}
for any $\sigma \in \Gal(\Qbar/\Qb)$.
\end{proposition}

\begin{proof}
Consider the unary lattice $N = (\Zb, 6x^2)$ with dual lattice $\frac{1}{12}\Zb$.
From Lemma 2.1 in \cite{lischwagenscheidt}, we have
$$
\langle \theta_N(\tau), \phi_\eta \rangle = 2 \eta(\tau), \quad \phi_\eta := \phi_{1/12}  - \phi_{5/12}  - \phi_{7/12}  + \phi_{11/12}  \in \Sc_N. 
$$
  Now for any $f \in M^!_{k, P}$, we have
\begin{align*}
  \langle f(\tau), \Theta_{ P}(\tau, h) \rangle_{\Pet}
  &= \frac{1}{4}\langle f_\eta, \Theta_{P}(\tau, h)\otimes\overline{\theta_N(\tau)} \rangle_{\Pet}
  = \frac{y_0^{1-k}}{4} \Phi_{L,k-1}(f_\eta,z_0, h)
\end{align*}
where $L = \Zb[6] \oplus P^-$ is a lattice of signature $(1, 2)$,
$z_0 = z_0^+ \in \Db_{L}$ is the negative 2-plane associated with $U^- \otimes \Rb \subset L \otimes \Rb$ and
$$
f_\eta := f(\tau) \otimes  \eta(\tau)^{-1}\phi_\eta \in M^!_{k-1/2, L^-}.
$$
Theorem \ref{shimura lift} tells us that the function $\Phi_{L, k-1}(f_\eta,z,h)$ is a meromorphic modular form of weight $2k-2$ on the Shimura variety $X_{K(L)}$ associated with $V := L \otimes \Qb$ with $K(L)$ an open compact stabilizing $\hat L$ and acting trivially on $\hat L'/\hat L$.
Furthermore, it has the same field of definition as that of $f_\eta$, hence also of $f$.

We can now write 
\begin{equation}
  \label{eq:Petratio}
  \begin{split}
      \frac{  \langle f(\tau), \Theta_{ \vr, P}(\tau) \rangle_{\Pet}}{  \| \Theta_{\vr, P}(\tau) \|^2_{\Pet}}
      &=
\Psi_{\vr, f}(z_0, 1),\\
\Psi_{\vr, f}(z, h)
&:=      \frac{ \sum_{h' \in \Cl(K_\f)} \overline{\vr(\overline{h'}/h')} \Phi_{L,k-1}(f_\eta,z, hh') }{
\sum_{h' \in \Cl(K_\f)} \overline{\vr(\overline{h'}/h')} \Phi_{L,k-1}((\Theta_{\vr, P})_\eta,z, hh') }
  \end{split}
\end{equation}
with $\Psi_{\vr, f}(z, h)$ a meromorphic modular function on $X_{K(L)}$ with poles away from $Z(U)$.
Furthermore,
$$
\Psi_{\vr, f}(z, 1) = \Psi_{\vr, f}(z, h) 
$$
for all $[z, h] \in X_{K(L)}$.
If $f$ is defined over $\Qbar$, so is $f_\eta$ and $\Psi_{\vr, f}$. Then the special value $\Psi_{\vr, f}(z_0, h)$ is in $\Qbar$. 
For $\sigma \in \Gal(\Qbar/\k)$ with image $\sigma_h \in {\Gal(\k^\ab/\k)}$ for $h \in \hat \k^\times$, we can apply Remark \ref{rmk:commutec} and equation \eqref{eq:rec} to obtain
\begin{equation*}
  \begin{split}
        \lp \frac{  \langle f(\tau), \Theta_{ \vr, P}(\tau) \rangle_{\Pet}}{  \| \Theta_{\vr, P}(\tau) \|^2_{\Pet}} \rp^{    \sigma_{}}
&=
 \Psi_{\vr, f}(z_0, 1)^{    \sigma_{}}
 = \Psi_{\vr^\sigma, f^\sigma}(z_0, h^{-1}) = \Psi_{\vr^\sigma, f^\sigma}(z_0, 1)\\
&=      \frac{  \langle f^\sigma(\tau), \Theta_{ \vr^\sigma, P}(\tau) \rangle_{\Pet}}{  \| \Theta_{\vr^\sigma, P}(\tau) \|^2_{\Pet}} .
  \end{split}
\end{equation*}
This holds similarly when $\sigma$ is complex conjugation $\c \in \GQ$, hence \eqref{eq:Petratio2} holds for all $\sigma \in \GQ$.
\end{proof}

\subsection{Comparing Petersson norms}
\label{subsec:Pnorm}
Now, we will compute $\|\theta_{U_a, k}\|^2_{\mathrm{Pet}}$ for $k \ge 2$ and compare them with $U_a$ varying.
The same calculations in the case of weight one newforms with CM by a real quadratic field are contained in \cite{Li18}.

First, we identify $U_a \oplus U_{-a}$ and $V_a := (M_2(\Qb), a \det)$ by
\begin{equation}
  \label{eq:UVid}
  \begin{split}
      (1, 1) &\mapsto \smat{2}{0}{0}{0}, \quad  (1, -1) \mapsto \smat{0}{0}{0}{2},\\
  (\sqrt{D}, \sqrt{D}) &\mapsto \smat{0}{2D}{0}{0}, \quad   (\sqrt{D}, -\sqrt{D}) \mapsto \smat{0}{0}{-2}{0}.
  \end{split}
\end{equation}
Notice we have the totally isotropic splitting
\begin{equation}
  \label{eq:isosplit}
  V_a = V_a^+ + V_a^-, \quad 
  V_a^\pm := \{(x, \pm x): x \in \k\} \subset U_a \oplus U_{-a}.
\end{equation}
Also, the linear map $x \mapsto \smat{a}{}{}{1} x$ gives an isometry from $V_1$ to $V_a$.
Let $L_1 \subset V_1$ be the unimodular lattice $M_2(\Zb)$, and denote $L_a := \smat{a}{}{}{1} L_1 \subset V_a$ the unimodular lattice in $V_a$.

Let $z_0 := \Rb \smat{1}{}{}{1} + \Rb \smat{}{1}{-D}{} \in \Db_{V_a}$  be the oriented negative 2-plane.
So we can write
$$
v^{k-1} \theta_{U_a, k}(\tau, \phi_1) \overline{\theta_{U_a, k}(\tau, \phi_2)}
= \theta_{V_a}(g_\tau, z_0, \phi^{(k, k)})
$$
with $\phi^{(k, k)} = \phi \phi^{(k, k)}_\infty$ for $\phi = \phi_1 \otimes \overline{\phi_2} \in  \Sc(\hat U_a; \Cb)\otimes  \Sc(\hat U_{-a}; \Cb) \cong   \Sc(\hat V_a; \Cb)$ and $\phi^{(k, k)}_{\infty}(z_1, z_2) = z_1^k \overline{z_2}^k \in \Sc(U_a(\Rb)) \otimes \Sc(U_{-a}(\Rb)) \cong \Sc(V_a(\Rb))$.
By the Siegel-Weil formula, we have
$$
2\int_{[\SL_2]} \theta_{V_a}(g, z, \phi^{(k, k)}) dg =
E^{H_{V_a}}([z, 1], \hat\phi^k)
$$
for $k \ge 2$, where $\hat\phi^k = \hat\phi \hat\phi_\infty^k \in \Sc((W \otimes V^+_a)(\Ab))$ is the Fourier transform of $\phi^{(k, k)}$, and $E^{H_{V_a}}$ is the Eisenstein series induced from the parabolic $P^+$ stabilizing $V^+_a$. Note that $E^{H_{V_a}}$ a modular form of weight $(2k, 0)$ that is nearly holomorphic in $z_1$ and holomorphic in $z_2$ on the orthogonal Shimura variety $X_{V_a, K}$ associated with $V_a$.
Note that the canonical isomorphism $H_{V_a} \cong H_{V_1}$ induces a canonical isomorphism
\begin{equation}
  \label{eq:XVa}
  X_{V_a, K} \cong X_{V_1, K}
\end{equation}
for any $a \in \Qb_{> 0}$.
Also, $X_{V_a, K}$ is isomorphic to disjoint unions of products of modular curves. 
We now have the following result.
\begin{proposition}
  \label{prop:Pnorm-compare}
  For every even integral lattice $P \subset U_a$ and Hecke character $\vr$ of weight $k \ge 2$,  there exists $\alpha_{\vr, P} \in \Qbar$ and $\Omega_{\vr, P} \in \Rb_{> 0}$ such that
  \begin{equation}
    \label{eq:Omega-chi}
      \alpha_{\vr^\sigma, P} =       \alpha_{\vr, P}^\sigma,\qquad
    \Omega_{\vr, P} =
      \frac{\| \Theta_{\vr, P}(\tau) \|^2_\Pet }{ \alpha_{\vr, P} }
      =    \frac{\| \Theta_{\vr^\sigma, P}(\tau) \|^2_\Pet }{ \alpha^\sigma_{\vr, P} }
      =     \Omega_{\vr^\sigma, P},
   \end{equation}
 for every $\sigma \in \Gal(\Qbar/\Qb)$  and
   \begin{equation}
     \label{eq:Omega-ratio}
     \frac{\Omega_{\vr, P}}{\Omega_{\vr}} \in \Qb_{> 0},
   \end{equation}
   where we denote $\Omega_\vr := \Omega_{\vr, \Oc_\k}$ with $\Oc_\k \subset U_1$. 
\end{proposition}

\begin{remark}
  \label{rmk:ratio}
  When $\Theta_{\vr, P}$ is replaced by a newform, the corresponding result follows from Deligne's conjecture for symmetric square $L$-function proved by Sturm \cite{Sturm1, Sturm2}. However, this does not directly imply Prop.\ \ref{prop:Pnorm-compare} since it is not clear whether the ratio of 
  Petersson norms of two arbitrary newforms is algebraic. In fact, it is unlikely that the analogous statement in this proposition holds for an arbitrary cusp form. 
\end{remark}

\begin{proof}
  For convenience, denote
  \begin{equation}
    \label{eq:omega-Chi-P}
    \omega_{\vr, P} := \|\Theta_{\vr, P}\|^2_\Pet.
  \end{equation}
  and $\omega_\vr := \omega_{\vr, \Oc_\k}$ for $\Oc_\k \subset U_1$. 
  The following lattice in $U_a \oplus U_{-a} \cong V_a$
  \begin{equation}
    \label{eq:L1}
    L := \{(x, y) \in P' \otimes (P^-)':x - y \in P_1\}  
  \end{equation}
  is unimodular in $V_a$ and equal to $\gamma \cdot L_a$  for some  $\gamma \in \SO(V_a)$. 
By conjugating the isomorphism in \eqref{eq:UVid} with $\gamma$, we can suppose that $L = L_a$. 
  
Now for any
$\phi, \phi' \in \Sc(\hat U_a)$, we have
  \begin{align*}
&\quad    \frac{\langle \theta_{\vr}(\tau, \phi), \theta_{\vr}(\tau, \phi') \rangle_{\Pet}}{\omega_{\vr, P}}\\
&=    \frac{\sum_{h_1, h_2 \in \Cl(K_\f)}
 \vr(\overline{h_1}/h_1)
\overline{ \vr(\overline{h_2}/h_2)}
 \langle \theta_{U_a, k}(\tau, \omega_a(h_1) \phi), \theta_{U_a, k}(\tau, \omega_a(h_2)\phi') \rangle_{\Pet}}{\sum_{h_1, h_2 \in \Cl(K_\f)} 
 \vr(\overline{h_1}/h_1)
\overline{ \vr(\overline{h_2}/h_2)}
 \sum_{\mu \in \df^{-1}/\Oc_\k} \langle \theta_{U_1, k}(\tau, \omega_1(h_1) \phi_{\mu}), \theta_{U_1, k}(\tau, \omega_1(h_2)\phi_{\mu}) \rangle_{\Pet}}\\
&=    \frac{\sum_{h_1, h_2 \in \Cl(K_\f)}
   \vr(\overline{h_1}/h_1)
\overline{ \vr(\overline{h_2}/h_2)}
  E^{H_{V_a}}([z_0, (h_1, h_2)], \widehat{\phi\otimes\phi'}^k)}
  {\sum_{h_1, h_2 \in \Cl(K_\f)}
   \vr(\overline{h_1}/h_1)
\overline{ \vr(\overline{h_2}/h_2)}
  E^{H_{V_1}}([z_0, (h_1, h_2)], \hat{\phi}^k)} \\
    & = \Psi_\vr(z_0, 1; \widehat{\phi\otimes\phi'}^k), 
  \end{align*}
  where $\phi_1 := \cha(\hat L_1) \in \Sc(\hat V_a)$ and
  \begin{align*}
      \Psi_\vr(z, h; \varphi)
    &:=     \frac{\sum_{h_1, h_2 \in \Cl(K_\f)}
   \vr(\overline{h_1}/h_1)
\overline{ \vr(\overline{h_2}/h_2)}
 E^{H_{V_a}}([z, h(h_1, h_2)], \varphi^k)}
      {\sum_{h_1, h_2 \in \Cl(K_\f)}
   \vr(\overline{h_1}/h_1)
\overline{ \vr(\overline{h_2}/h_2)}
      E^{H_{V_1}}([z, h(h_1, h_2)], \hat \phi_1^k)}
  \end{align*}
  for $\varphi \in \Sc(\widehat{V_1^+ \otimes W})$.
  Standard calculations give that $\hat\phi_1 = \cha(M_2(\hat\Zb)) \in \Sc(\widehat{V_1^+ \otimes W})$.
  So $\Psi_{\vr}$ is a ratio of nearly holomorphic modular forms on $X_{V_a, K} \cong X_{V_1, K}$ defined over $\Qb(\vr, \widehat{\phi\otimes\phi'})$. Its values at CM points are defined over $\k^{\mathrm{ab}}(\vr, \widehat{\phi   \otimes \phi'})$, and reciprocity law still applies \cite{shimura75}. Arguing as in the proof of Proposition~ \ref{prop:algGal}, we have
  $$
  \Psi_\vr(z_0, 1; \varphi)^\sigma = \Psi_{\vr^\sigma}(z_0, 1; \varphi^\sigma)
  $$
  for all $\sigma \in \GQ$.
  This gives us
  \begin{equation}
    \label{eq:Pet-rat-Gal}
    \lp \frac{\omega_{\vr, P}}{\omega_{\vr}} \rp^\sigma =
     \frac{\omega_{\vr^\sigma, P}}{\omega_{\vr^\sigma}}.
   \end{equation}
for all $\sigma \in \Aut(\Cb/\Qb)$,   since   $\hat\phi_a = \cha(M_2(\hat\Zb))$ is $\Qb$-valued when $\phi_a = \cha(\hat L_a)$.
   In particular, $\frac{\omega_{\vr, P}}{\omega_{\vr}} \in \Qb(\vr)$.

   By a similar argument, we can show that $\frac{\omega_{\vr^\sigma, P}}{\omega_{\vr, P}} \in \Qb(\vr)$ and satisfies
   \begin{equation}
     \label{eq:Pet-rat2}
\lp     \frac{\omega_{\vr^\sigma, P}}{\omega_{\vr, P}}\rp^{\sigma'}
     =      \frac{\omega_{\vr^{\sigma\sigma'}, P}}{\omega_{\vr^{\sigma'}, P}}
   \end{equation}
   for all $\sigma, \sigma' \in \GQ$.
   Therefore, the function
   $$\sigma \mapsto      \frac{\omega_{\vr^{\sigma^{-1}}, P}}{\omega_{\vr, P}}$$
   defines a 1-cocycle in $H^1(\Gal(F/\Qb), F^\times)$ with $F$ the Galois closure of $\Qb(\vr)$.
   By Hilbert's Theorem 90, this cohomology group is trivial, and there exists $\alpha_{\vr, P} \in F^\times$ such that
   $$
   \frac{\omega_{\vr^{\sigma}, P}}{\omega_{\vr^{}, P}}
   =      \frac{\alpha_{\vr^{}, P}^\sigma}{\alpha_{\vr^{}, P}}
   $$
   for all $\sigma \in \Gal(F/\Qb)$.
   Simple manipulations give us
   $\frac{\alpha_{\vr^\tau, P}^\sigma}{\alpha_{\vr^\tau, P}} = \frac{\alpha_{\vr, P}^{\tau\sigma}}{\alpha^\tau_{\vr, P}}$
   for all $\sigma, \tau \in \Gal(F/\Qb)$.
   So $\frac{\alpha_{\vr^\tau, P}}{\alpha^\tau_{\vr, P}}$ is rational, and we can choose $\alpha_{\vr^\sigma, P}$
   to satisfy the first equation in \eqref{eq:Omega-chi}
      for all $\sigma \in \Gal(F/\Qb)$.
   Setting 
   $\Omega_{\vr, P} := \omega_{\vr, P}/\alpha_{\vr, P}$ then proves the second equation in \eqref{eq:Omega-chi}.
      Substituting  \eqref{eq:Omega-chi} into \eqref{eq:Pet-rat-Gal} then proves \eqref{eq:Omega-ratio}.
\end{proof}

Finally, we state a lemma comparing the constant $\Omega_\vr$ and the Petersson norm of the newform $\vartheta_\vr$. The proof uses the harmonic Maass form that we construct in the proof of Theorem \ref{thm:Gal} in Section~\ref{subsec:preimage} below, so we postpone the proof until then.
It would be interesting to give a more direct proof.
\begin{lemma}
  \label{lemma:ratio}
  For an algebraic Hecke character $\vr$, let $\vartheta_\vr$ be the newform defined in \eqref{eq:vartheta} and $\Omega_\vr \in \Rb_{> 0}$ the quantity from Proposition~\ref{prop:Pnorm-compare}.
  Then 
  $ \|\vartheta_\vr\|_\Pet^2/\Omega_\vr \in \Qbar$
  satisfies
  \begin{equation}
    \label{eq:Gal-ratio}
\lp  \|\vartheta_\vr\|_\Pet^2/\Omega_\vr\rp^\sigma
  =
  \|\vartheta_{\vr^\sigma}\|_\Pet^2/\Omega_{\vr^\sigma}
  \end{equation}
  for all $\sigma \in \GQ$.
  In particular, it is   in $F_\vr$. 
\end{lemma}

\subsection{Preimages of CM newforms}
\label{subsec:preimage}
Using a rational basis of $M^!_{k, P}$, applying Proposition~\ref{normalized harmonic Maass form} and Proposition~\ref{prop:Pnorm-compare}, we can prove the following result.

\begin{theorem}
  \label{thm:hmfchi}
  Let $P \subset U$ be a lattice, $\vr$ an algebraic Hecke character, and  $\Omega_{\vr} \in \Rb_{> 0}$ be as in Proposition \ref{prop:Pnorm-compare}.
Then there exists
 a vector-valued harmonic Maass form $\tilde\Theta_{\vr, P} = (\tilde \theta_{\vr, P + \mu})_{\mu \in P'/P}$ such that
 $\xi_{2-k} \tilde\Theta_{\vr, P} =  \Theta_{\vr, P}/ \Omega_\vr$ and the holomorphic part $\tilde\Theta^+_{\vr, P}$ has Fourier coefficients in $\Qb(\vr)$ satisfying
 \begin{equation}
   \label{eq:Galsig}
   \lp \tilde\Theta^+_{\vr, P} \rp^\sigma  =    \tilde\Theta^+_{\vr^\sigma, P} 
 \end{equation}
 for all $\sigma \in \Gal(\Qbar/\Qb)$. 
\end{theorem}
\begin{proof}
  Let $f_{m, \mu} \in M^!_{k, P}$ be the basis in Lemma~\ref{basis modular forms}. 
  Then
  \begin{equation}
    \label{eq:tTheta}
    \tilde{\Theta}^+_{\vr, P}(\tau) :=
\frac{\Omega_{\vr, P}}{\Omega_\vr}
    \frac{\|\Theta_{\vr, P}\|_\Pet^2}{\Omega_{\vr, P}}
    \sum_{\mu \in L'/L}\sum_{m \in \Q}
  \frac{\langle f_{m,\mu}, \Theta_{\vr, P}\rangle_{\Pet}}{\|\Theta_{\vr, P}\|_\Pet^2}
  q^m\phi_\mu
  \end{equation}
  satisfies the conditions by Proposition~\ref{prop:Pnorm-compare} and equation \eqref{eq:Petratio2}.
\end{proof}

 Let $\tilde\Theta^+_{\vr, P_j} = (\tilde \theta^+_{\vr, P_j + \mu})_{\mu \in P'_j/P_j}$ be the mock modular form in the above theorem.
  Then by \eqref{eq:vv2sc2} and Theorem \ref{thm:hmfchi}, the harmonic Maass form 
\begin{equation}
  \label{eq:tvar}
  \tilde g_\vr(\tau) :=
  \frac{1}{ |\Oc_\k^\times|}
  \sum_{1 \le i  \le \ell} \overline{\vr(t_{i })}^{-1}
  \sum_{\mu \in (\Oc_\k/\mf)^\times}\overline{\eta(\mu)}
  \tilde\theta_{\vr, P_{i } + \mu}(\tau)  
\end{equation}
satisfies $\xi_{2-k}( \tilde g_\vr) = \vartheta_\vr(\tau)/\Omega_\vr$ and
\begin{equation}
  \label{eq:Galsc}
(\tilde g_{\vr}^+)^\sigma =
\tilde g_{\vr^\sigma}^+,
\end{equation}
which implies that its holomorphic part Fourier coefficients at the cusp infinity are in $\Qb(\vr)$.
By \cite[Proposition~4.8]{scheithauerweil}, we know that $\tilde g_\vr$ has level $\Gamma_0(N)$ and nebentypus character $\chi_\k \overline\eta$.
We first use this scalar-valued modular form to prove Lemma \ref{lemma:ratio}.

\begin{proof}[Proof of Lemma \ref{lemma:ratio}]
Let $\tilde{g}_\vr$ be as in \eqref{eq:tvar}. 
Then this ratio can be expressed as
\begin{align*}
  \|\vartheta_\vr\|_\Pet^2/\Omega_\vr 
  &= \langle \vartheta_\vr, \vartheta_\vr \rangle/\Omega_\vr
  =
\text{Constant term of }
\lp \tr^N_1   (\tilde g_\vr \cdot \vartheta_\vr)\rp^+.        
\end{align*}
We will prove more generally that
\begin{equation}
  \label{eq:Gal-act}
\lp \lp \tr^N_1  \vartheta_\vr(\tau) \tilde g_\vr( \tau)\rp^+\rp^\sigma
=  \lp \tr^N_1  \vartheta_{\vr^\sigma}(\tau) \tilde g_{\vr^\sigma}( \tau)\rp^+
\end{equation}
for all $\sigma \in \GQ$. 
Using the definition of $\tilde g_\vr$ and \eqref{eq:vv2sc2}, we have
\begin{align*}
  \tr^N_1  \vartheta_\vr(\tau) \tilde g_\vr( \tau)
  &= 
     \frac{1}{|\Oc_\k^\times|^2}
    \sum_{\substack{ 1 \le i, i'  \le \ell\\ \mu, \mu' \in (\Oc_\k/\mf)^\times}}
\frac{\overline{\eta(\mu')}{\eta^{}(\mu)}}{  \overline{\vr(t_{i' })}
  \vr^{}(t_{i })}
  \tr^N_1
\lp \theta_{\vr, P_{i } + \mu}(\tau)  \tilde\theta_{\vr^{}, P_{i' } + \mu'}(\tau)\rp.
\end{align*}
Since $\theta_{\vr^\sigma, P_i + \mu}$ and $\tilde\theta_{\vr^{\sigma'}, P_{i'} + \mu'}$ are components of vector-valued modular forms $\Theta_{\vr^\sigma, P_i}(\tau) \in S_{k, P_i}$ and  $\tilde\Theta_{\vr^{\sigma'}, P_{i'}}(\tau) \in H_{2-k, P^-_{i'}}$ respectively, we have
\begin{align*}
    \tr^N_1
\lp \theta_{\vr, P_{i } + \mu}(\tau)  \tilde\theta_{\vr^{}, P_{i' } + \mu'}(\tau)\rp
  &=
    \sum_{\gamma \in \Gamma_0(N) \backslash \SL_2(\Zb)}
    \langle \Theta_{\vr, P_{i }}(\tau)  \otimes \tilde\Theta_{\vr^{}, P_{i' } }(\tau), \phi_{\mu} \otimes \phi_{\mu'}\rangle \mid_2 \gamma    \\
  &=
    \left\langle \Theta_{\vr, P_{i }}(\tau)  \otimes \tilde\Theta_{\vr^{}, P_{i' } }(\tau),
\tr^N_1 (\phi_\mu \otimes \phi_{\mu'})
    \right\rangle,
\end{align*}
where $\tr^N_1(\phi_{\mu} \otimes \phi_{\mu'})$ is defined in \eqref{eq:tr} and has value in $\Qb$ by the discussion there.

Given $\sigma \in \GQ$, we can apply Remark \ref{rmk:commutec} and Theorem \ref{thm:hmfchi} to obtain
  \begin{align*}
&\lp 
\frac{\overline{\eta(\mu')}{\eta^{}(\mu)}}{  \overline{\vr(t_{i' })}
  \vr^{}(t_{i })}
  \tr^N_1
    \lp \theta_{\vr, P_{i } + \mu}(\tau)  \tilde\theta_{\vr^{}, P_{i' } + \mu'}(\tau)\rp^+    \rp^{\sigma}\\
    &=
\frac{\overline{\eta^{\sigma}(\mu')}{\eta^{\sigma}(\mu)}}{  \overline{\vr^{\sigma}(t_{i' })}
  \vr^{\sigma}(t_{i })}     
    \left\langle \Theta_{\vr, P_{i }}(\tau)  \otimes \tilde\Theta^+_{\vr^{}, P_{i' } }(\tau),
\tr^N_1(\phi_{\mu} \otimes \phi_{\mu'})\right\rangle^{\sigma}\\
    &=
\frac{\overline{\eta^{\sigma}(\mu')}{\eta^{\sigma}(\mu)}}{  \overline{\vr^{\sigma}(t_{i' })}
  \vr^{\sigma}(t_{i })}      
\left\langle \Theta_{\vr^{\sigma}, P_{i }}(\tau)  \otimes \tilde\Theta^+_{\vr^{\sigma}, P_{i' } }(\tau), 
\tr^N_1(\phi_{\mu} \otimes \phi_{\mu'}) \right\rangle\\
&=     
\frac{\overline{\eta^\sigma(\mu')}{\eta^{\sigma'}(\mu)}}{  \overline{\vr^\sigma(t_{i' })}
  \vr^{\sigma'}(t_{i })}
  \tr^N_1
    \lp \theta_{\vr^\sigma, P_{i } + \mu}(\tau)  \tilde\theta_{\vr^{\sigma'}, P_{i' } + \mu'}(\tau)\rp^+
  \end{align*}
  for any $i, i', \mu, \mu'$.
  Therefore, $\lp \tr^N_1  \vartheta_\vr(\tau) \tilde g_1( \tau)\rp^+$ satisfies \eqref{eq:Gal-act}, and equation \eqref{eq:Gal-ratio} holds.
  The last claim then follows from the second equality in \eqref{eq:Omega-chi}. 
\end{proof}


Similarly, we have the following result concerning the Fourier expansion of $(\tilde g_\vr)^\sigma$ at various cusps.

\begin{proposition}
  \label{prop:Galexp}
  Let $\tilde g_\vr$ be the harmonic Maass form of level $N$ defined in \eqref{eq:tvar}.
  For any $\gamma \in \SL_2(\Zb)$ and $\sigma \in \GQ$, we have
  \begin{equation}
    \label{eq:Galexp}
  (\tilde g_\vr \mid_{2-k} \gamma)^{\sigma}
=    \tilde g_{\vr^\sigma} \mid_{2-k} \gamma^a,
\end{equation}
where $\sigma \mid_{\Gal(\Qb(\zeta_N)/\Qb)} = \sigma_a$ for $a \in \Zb/N\Zb$ and $\gamma^a$ is defined in \eqref{eq:gammaa}
\end{proposition}

\begin{remark}
  \label{rmk:shim}
  When $\tilde g_\vr$ is replaced by any modular function $f$ of level $\Gamma(N)$ defined over $\Qbar$, equation \eqref{eq:Galexp} holds with $\tilde g_{\vr^\sigma}$ replaced by $f^\sigma$ \cite[Theorem 6.23]{shimauto}.
Proposition \ref{prop:Galexp} can be considered as a generalization of this classical result to the setting of harmonic Maass forms.
\end{remark}
\begin{proof}
We can apply Remark \ref{rmk:commutec}, equation \eqref{eq:gammaa} and Theorem \ref{thm:hmfchi} to obtain
  \begin{align*}
  (\tilde g_\vr \mid_{2-k} \gamma)^{\sigma}
    &=
      \frac{1}{ |\Oc_\k^\times|}
    \sum_{\substack{1 \le i  \le \ell\\ \mu \in (\Oc_\k/\mf)^\times}}
\lp  \overline{\vr(t_{i })}^{-1}
                   \overline{\eta(\mu)}
  \tilde\theta_{\vr, P_{i } + \mu} \mid_{2-k} \gamma) \rp^{\sigma}\\
  &=
      \frac{1}{ |\Oc_\k^\times|}
    \sum_{    1 \le i  \le \ell}
\overline{\vr^{\sigma }(t_{i })}^{-1}
  \left\langle  
 \rho_{P_i}(\gamma)^{\sigma_a} \tilde\Theta_{\vr^{\sigma}, P_{i }},  \sum_{\mu \in (\Oc_\k/\mf)^\times}                    \overline{\eta^{\sigma}(\mu)} \phi_\mu
  \right\rangle\\
  &=  \tilde g_2^{\sigma_a} \mid_{2-k} \gamma^a.
  \end{align*}
  This finishes the proof.
\end{proof}
Now, we are ready to produce a good harmonic Maass form for $\vartheta_\vr$ satisfying Theorem \ref{thm:Gal}.
\begin{proof}[Proof of Theorem~\ref{thm:Gal} and Corollary \ref{cor:main}]
  To make $\tilde{g}_\vr$ in \eqref{eq:tvar} a good form, we first move all its poles to the cusp $\infty$ in the following way. 
Since the orthogonal complement of $\vartheta_\vr$ in $S_k(N, \chi)$ has a basis defined over $\Qbar$, 
we can apply Serre duality to find $f_\vr \in M^!_{2-k}(N, \overline\chi)$ such that $\prin(f_\vr) \in \Qbar[q^{-1}]$ and $\tilde g_\vr^+ - f_\vr$ is holomorphic away from the cusp $\infty$ (see the discussion after Proposition~\ref{normalized harmonic Maass form}). 
Now for any $\sigma \in \GQ$, we define
$$
f_{\vr^\sigma} := (f_\vr)^\sigma. 
$$
Suppose $\sigma \mid_{\Gal(\Qb(\zeta_N)/\Qb)} = \sigma_a$ for $a \in \Zb/N\Zb$.
By Proposition \ref{prop:Galexp} and Remark \ref{rmk:shim}, we have
\begin{align*}
  (\tilde g_{\vr^\sigma} - f_{\vr^\sigma}) \mid_{2-k} \gamma
  &=   (\tilde g_\vr^{\sigma} - f_\vr^{\sigma}) \mid_{2-k} \gamma
    =  ((\tilde g_\vr - f_\vr) \mid_{2-k} \gamma^{a^{-1}})^{\sigma}
\end{align*}
 for $\gamma \in \SL_2(\Zb)$, 
 which implies that $\tilde g_{\vr^\sigma}^+ - f_{\vr^\sigma}$ is also holomorphic away from $\infty$.
 By applying Proposition \ref{prop:Pnorm-compare} and Lemma \ref{lemma:ratio}, we see that the harmonic Maass form
 $$
 \tilde\vartheta_\vr :=
 \frac{1}{|\Gal(\Qb(\vr)/F_\vr)|}
 \frac{\Omega_\vr}{   \|\vartheta_{\vr}\|_\Pet^2}
 \sum_{\sigma \in \Gal(\Qb(\vr)/F_\vr)} \tilde g_{\vr^\sigma} - f_{\vr^\sigma}
 $$
 is good for $\vartheta_\vr$ and satisfies the theorem.
 Furthermore, all its Fourier coefficients are in $F_\vr$, which proves the corollary.
\end{proof}

\section{A numerical example}
\label{sec:eta8}
In this section we illustrate our construction of $\xi$-preimages of CM newforms in a numerical example. We consider the unique normalized newform 
\[
g = \eta(3z)^8 = q - 8 q^4 + 20 q^7 - 70 q^{13} + 64 q^{16} + 56 q^{19}+ \dots\in S_4(9)
\]
of weight $4$ for $\Gamma_0(9)$. We have seen in Example~\ref{ex:eta8} that it has CM by $\k = \Q(\sqrt{-3}) = \Q(\zeta)$, where $\zeta = \zeta_6 = e^{2\pi i /6}$. Moreover, $\eta(z)^8$ appears as a component of the vector-valued binary theta series $\Theta_{P,4}$ of weight $k=4$ for the lattice $P = (\O_\k,\Nm)$. We will construct a $\xi$-preimage for $\Theta_{P,4}$, which yields a $\xi$-preimage for $g$ as well.


Since $g$ has rational Fourier coefficients, our construction will yield a harmonic Maass form $\tilde{g} \in H_{-2}(9)$ with $\xi_{-2}\tilde{g} = g/\|g\|_{\Pet}^{2}$ and \emph{rational} Fourier coefficients in the holomorphic part. Such a harmonic Maass form was also constructed by Bruinier, Ono, and Rhoades in \cite{bruinieronorhoades} using a different method. At the end of this section we will compare our construction to the harmonic Maass form from \cite{bruinieronorhoades} in order to validate our results.

%
As $\Qb(\vr) = \Qb$, we can choose $\Omega_\vr = \|\Theta_{P, 4}\|_{\Pet}^2$, which is given in \eqref{eq:P4norm}.
By Proposition~\ref{normalized harmonic Maass form}, the generating series
\begin{align}\label{thetaP xi preimage}
\tilde{\Theta}_{P,4}^+(z) = \sum_{\mu \in P'/P}\sum_{m \in \Q}\frac{\langle f_{m,\mu},\Theta_{P,4}\rangle_{\Pet}}{\langle \Theta_{P,4},\Theta_{P,4}\rangle_{\Pet}}q^m \phi_\mu
\end{align}
is the holomorphic part of a harmonic Maass form $\tilde{\Theta}_{P,4} \in H_{-2,P^-}$ with $\xi_{-2}\tilde{\Theta}_{P,4} = \Theta_{P,4}/\|\Theta_{P,4}\|_{\Pet}^{2}$. Here the forms $f_{m,\mu}$ are the generating set of $M_{4,P}^!$ from Lemma~\ref{basis modular forms}. In our case, the space $M_{4,P}$ of holomorphic modular forms is one-dimensional, and spanned by $\Theta_{P,4}(z)$. In particular, the forms $f_{m,\mu}$ are given by $f_{m,0} = 0$ and $f_{m,\pm 1} = \pm f_m$ where
\begin{align*}
f_m(\tau) = \frac{1}{2}q^{-m}(\phi_{1}-\phi_{-1})+O(q^{4/3}), \qquad (m \geq -1/3, \ m \equiv 2/3 \pmod 1).
\end{align*}
Explicitly, for $m = -1/3$ we have 
\[
f_{-1/3}(\tau) = \frac{1}{2}\eta(\tau)^8(\phi_{1}+\phi_{-1})
\]
and for $m \geq 2/3$ with $m \equiv 2/3 \pmod 1$ we can write
\[
f_n(\tau) = \frac{1}{2}P_{n+1/3}(j(\tau))\cdot\eta(\tau)^8 (\phi_1-\phi_{-1})
\]
with certain polynomials $P_{n+1/3}(X)$, the first few of which are given by
\begin{align*}
  P_{0}(X) &= 1, \quad 
  P_{1}(X) = X -736, \quad
  P_{2}(X) = X^2 -1480X+153860.
\end{align*}

\subsection{Petersson inner products as special values of theta lifts} Now we evaluate the inner product $\langle f_{m,\mu},\Theta_{P,4}\rangle_{\Pet}$ appearing in \eqref{thetaP xi preimage}. By symmetry, it suffices to do this for $\mu = 1$, that is, to compute $\langle f_{m},\Theta_{P,4}\rangle_{\Pet}$ for $m \geq -1/3$ with $m \equiv 2/3 \pmod 1$. In the proof of Proposition~\ref{prop:algGal} we evaluated these inner products by multiplying $f_m$ with $\eta^{-1}(\tau)\phi_\eta$ and thereby viewing $\langle f_{m},\Theta_{P,4}\rangle_{\Pet}$ as a special value of the Borcherds-Shimura theta lift of $f_m \otimes \eta^{-1}\phi_\eta$ on the signature $(1,2)$ lattice $\Z[6] \oplus P^-$, where $\Z[6] = (\Z,6x^2)$. Unfortunately, in our example the lattice $\Z[6] \oplus P^-$ is anisotropic, so we cannot directly work with the Fourier expansion of the Borcherds-Shimura lift. Hence it will be more convenient to choose a different function than $\eta^{-1}(\tau)\phi_\eta$, which yields a theta lift on an \emph{isotropic} lattice of signature $(1,2)$. The functions and lattices chosen below are special cases of a more general construction given in Section~7.1 \cite{bruinieryang} and Sections 2.7 and 4.1 in \cite{ehlenduke}, which also work for other positive definite binary lattices $P$.

Consider the (scalar-valued) weakly holomorphic modular form
\begin{align*}
\mathcal{F}(\tau)&=\frac{\eta(\tau)^{10}}{\eta(2\tau)^5\eta(4\tau)^6}+20\frac{\eta(\tau)^2 \eta(4\tau)^2}{\eta(2\tau)^5} \\
&\quad= q^{-1} + 10 - 64q^3 + 108q^4 - 513q^7 + 808q^8 - 2752q^{11} + 4016q^{12} + \dots
\end{align*}
of weight $-1/2$ for $\Gamma_0(4)$ in the Kohnen plus space. It is the unique such form with Fourier expansion $q^{-1} + O(1)$. By splitting it into the two parts $\mathcal{F}_0(\tau)$ and $\mathcal{F}_1(\tau)$ consisting of indices $n$ with $n \equiv 0 \pmod 4$ and $n \equiv -1 \pmod 4$, respectively, and then replacing $\tau$ by $\tau/4$, we obtain a vector-valued modular form $\mathcal{F}_0(\tau/4)\phi_0 + \mathcal{F}_1(\tau/4)\phi_1$ of weight $-1/2$ for the Weil representation of $\Z[1]$. By a slight abuse of notation, we denote this vector-valued form by $\mathcal{F}(\tau)$ as well. Let $\theta(\tau) = \sum_{r(2)}\sum_{n \equiv r (2)}q^{n^2/4}\phi_r$ denote the unary theta function of weight $1/2$ for the dual Weil representation of $\Z[1]$. Then we have
\begin{align}\label{ftheta12}
\langle \mathcal{F}(\tau), \theta(\tau) \rangle = 12,
\end{align}
which easily follows from the fact that the left-hand side is a (scalar-valued) holomorphic modular form of weight $0$ for $\SL_2(\Z)$, hence constant (this result is a special case of Lemma~4.7 in \cite{ehlenduke}).

Following Section~7.1 in \cite{bruinieryang} we consider the (isotropic) lattice
\[
L = \left\{\begin{pmatrix}b & -a/3 \\ c & -b\end{pmatrix}: a,b,c \in \Z\right\}
\]
with the quadratic form $Q(X) = 3\det(X)$. Its discriminant group is isomorphic to $\Z/6\Z$, and $\Gamma_0(3)$ acts on it by conjugation. The vector
\[
\lambda_0 = \begin{pmatrix}1/2 & -1/3 \\ 1 & -1/2\end{pmatrix}
\]
lies in $L'$ and has norm $Q(\lambda_0) = 1/4 > 0$, hence $U = \lambda^\perp$ defines a negative definite plane, that is, an element of the Grassmannian $\D$. Under the identification $\H \cong \D$ given in \eqref{identification DH} it corresponds to the point
\[
z_U = \frac{3+i\sqrt{3}}{6} \in \H.
\]
We consider the positive and negative definite sublattices
\begin{align*}
\mathcal{P} &= L \cap U^\perp = 2\lambda_0 \Z, \qquad
\mathcal{N} = L \cap U = \left\{\begin{pmatrix}b & (a-b)/3 \\ a+2b & -b\end{pmatrix}:a,b \in \Z\right\}, 
\end{align*}
of $L$. We see that $\mathcal{P}$ is isometric to $\Z[1] = (\Z,x^2)$ and $\mathcal{N}$ is isometric to $(\Z^2,-a^2-ab-b^2) = P^-$. Moreover, $L$ splits as an orthogonal direct sum 
\[
L = \mathcal{P} \oplus \mathcal{N} \cong \Z[1] \oplus P^-.
\]
This is the desired embedding of $P^-$ into an isotropic lattice of signature $(1,2)$. 

A short computation shows that
\[
y_U^{-2}\left(\lambda,\begin{pmatrix}-\overline{z}_U & \overline{z}_U^2 \\ -1 & \overline{z}_U \end{pmatrix}\right) = 12(a \zeta + b)
\]
for $\lambda = 2n\lambda_0 + \left(\begin{smallmatrix}b & (a-b)/3 \\ a+2b & -b\end{smallmatrix}\right) \in L = \mathcal{P} \oplus \mathcal{N}$, which implies the splitting of the theta kernel (defined in \eqref{eq:thetaLk})
\begin{align}\label{splitting theta}
\frac{1}{12^{3}}\Theta_{L,3}(\tau,z_U) = \theta(\tau) \otimes v^4\overline{\Theta_{P,4}(\tau)}.
\end{align}

We can now evaluate the inner products $\langle f_{m},\Theta_{P,4}\rangle$ using \eqref{ftheta12} and \eqref{splitting theta},
\begin{align*}
\langle f_{m},\Theta_{P,4}\rangle_{\Pet} &= \int_{\SL_2(\Z)\backslash \H}^{\reg}\langle f_m(\tau),\overline{\Theta_{P,4}(\tau)}\rangle v^{4}d\mu(\tau) \\
&= \frac{1}{12}\int_{\SL_2(\Z)\backslash \H}^{\reg}\langle f_m(\tau)\otimes \mathcal{F}(\tau),\theta(\tau) \otimes v^4\overline{\Theta_{P,4}(\tau)}\rangle d\mu(\tau) \\
&= \frac{1}{12^4}
\int_{\SL_2(\Z)\backslash \H}^{\reg}\langle f_m(\tau)\otimes \mathcal{F}(\tau),\Theta_{L,3}(\tau,z_U)\rangle d\mu(\tau) = \frac{1}{12^4}
\Phi_{L,3}(f_m\otimes \mathcal{F},z_U).
\end{align*}
Hence it remains to evaluate the Borcherds-Shimura theta lift $\Phi_{L,3}$ of $f_m\otimes \mathcal{F}$ at $z_U$. The advantage of the lattice $L$ is that the theta lift can be viewed as a meromorphic modular form of weight $6$ for $\Gamma_0(3)$ whose Fourier expansion for $y \gg 0$ large enough takes the simple form (see \cite{borcherds}, Theorem 14.3)
\begin{align}\label{theta lift fourier expansion}
\frac{1}{2}(-6i)^{-3}\Phi_{L,3}(f,z) = \sum_{n = 1}^\infty n^{3}\sum_{d \mid n}d^{-3}c_f(d^2/12, d)q^n,
\end{align}
where we identified $L'/L \cong \Z/6\Z$.

\subsection{Evaluation of theta lifts}

Let us now compute the inner product $\langle f_{-1/3},\Theta_{P,4}\rangle_{\Pet}$. It is given by a special value of the Borcherds-Shimura lift of the cusp form
\begin{align*}
f_{-1/3}(\tau) \otimes \mathcal{F}(\tau) &= \frac{1}{2}\bigg((\eta(\tau)^8 \mathcal{F}_1(\tau/4)(\phi_1-\phi_{5}) - \eta(\tau)^8 \mathcal{F}_0(\tau/4)(\phi_2-\phi_4)\bigg) \\
&=\frac{1}{2}\bigg((q^{\frac{1}{12}}  - 72q^{\frac{13}{12}} + 19q^{\frac{25}{12}} + \dots)(\phi_1-\phi_{5}) \\
&\qquad \quad -(10q^{\frac{4}{12}} + 28q^{\frac{16}{12}} + 144q^{\frac{28}{12}}+\dots)(\phi_2-\phi_{4})\bigg) \in S_{7/2,L}.
\end{align*}
Hence, by \eqref{theta lift fourier expansion} its Borcherds-Shimura lift has the Fourier expansion
\[
\frac{1}{2}(-6i)^{-3}\Phi_{L,3}(f_{-1/3} \otimes \mathcal{F},z) = \frac{1}{2}(q-6q^2+9q^3+4q^4+6q^5+\dots) = \frac{1}{2}\eta(z)^6\eta(3z)^6
\]
where the second identity follows from the fact that the Shimura lift is a cusp form of weight $6$ for $\Gamma_0(3)$, and this space is one-dimensional and spanned by $\eta(z)^6\eta(3z)^6$. We evaluate this cusp form at $z_U = \frac{3+i\sqrt{3}}{6}$ and obtain
\[
\eta(z_U)^6\eta(3z_U)^6 = -0.36019264\dots = -3\sqrt{3}\Omega_{-3}^6
\]
where
\[
\Omega_{-3} = \frac{1}{\sqrt{6\pi}}\left(\frac{\Gamma(1/3)}{\Gamma(2/3)}\right)^{\frac{3}{2}} = 0.64092738\dots
\]
is the Chowla-Selberg period of $\Q(\sqrt{-3})$. Taking everything together, we obtain the evaluation
\[
\langle f_{-1/3},\Theta_{P,4}\rangle_{\Pet} = \frac{1}{12^4}\Phi_{L,3}(f_{-1/3} \otimes \mathcal{F},z_U)=  -\frac{\sqrt{3}i}{32}\Omega_{-3}^6.
\]
Recalling the relation $\Theta_{P,4} = \frac{2i}{\sqrt{3}}f_{-1/3}$, this also yields the Petersson norm
\begin{equation}
  \label{eq:P4norm}
\|\Theta_{P,4}\|_{\Pet}^2 = \langle\Theta_{P,4},\Theta_{P,4}\rangle_{\Pet} = \frac{2i}{\sqrt{3}}\langle f_{-1/3},\Theta_{P,4}\rangle_{\Pet} =  \frac{1}{16}\Omega_{-3}^6.  
\end{equation}

Let us now compute the inner product $\langle f_{2/3},\Theta_{P,4} \rangle_{\Pet}$, which is more involved. Using the above Fourier expansion of $\mathcal{F}(\tau)$, we find
\begin{align*}
f_{2/3}(\tau)\otimes \mathcal{F}(\tau) &= \frac{1}{2}(q^{-\frac{11}{12}} - 64q^{\frac{1}{12}} + 196327q^{\frac{13}{12}} + 7318336q^\frac{25}{12} + \dots)(\phi_1-\phi_{5}) \\
&\quad-\frac{1}{2}(10q^{-\frac{8}{12}} + 108q^{\frac{4}{12}} + 1969208q^{\frac{16}{12}} + 220451216q^{\frac{28}{12}}+\dots)(\phi_2-\phi_{4})
\end{align*}
from which we can compute the Fourier expansion of the Borcherds-Shimura lift
\begin{align*}
\frac{1}{2}(-6i)^{-3}\Phi_{L,3}(f_{2/3}\otimes \mathcal{F},z) = -32q - 182q^2 - 288q^3 + 983876q^4 - 3659968q^5 + \dots
\end{align*}
Unfortunately, this Fourier expansion is only valid for $\Re(z) \gg 0$ large enough, so we cannot just plug in $z_U$ into this expansion. Hence, we want to write the lift as a rational function in weakly holomorphic modular forms. The Borcherds-Shimura lift is a meromorphic modular form of weight $6$ for $\Gamma_0(3)$ with poles of order $3$ at the CM points in $\mathcal{Q}_{3,-11,\pm 1}$ and $\mathcal{Q}_{3,-8,\pm 2}$. The latter fact can be read off from the exponents in the principal part of $f_{2/3}(\tau)\otimes \mathcal{F}(\tau)$. Since $-8$ and $-11$ have class number $1$, it follows from Proposition~I.1 in \cite{gkz} that the corresponding binary quadratic forms are given by
\[
\mathcal{Q}_{3,-11,\pm 1}/\Gamma_0(3) = \{[3,1,1], [3,-1,1]\}, \quad \mathcal{Q}_{3,-8,\pm 2}/\Gamma_0(3) = \{[3,2,1],[3,-2,1]\}.
\]

The weakly holomorphic modular form 
\[
j_3(z) = \frac{\eta(z)^{12}}{\eta(3z)^{12}} = q^{-1}- 12 + 54q - 76q^2 - 243q^3 + 1188q^4 - 1384q^5 - 2916q^6+\dots
\]
is a Hauptmodul for $\Gamma_0(3)$, and its values at the CM points corresponding to the above quadratic forms are given by
\[
j_3(z_{[3,\pm 1,1]}) = 5\pm 8\sqrt{11}i,  \quad j_3(z_{[3,\pm 2,1]}) = -23 \pm 5\sqrt{8}i.
\]

Now, if we divide the Borcherds-Shimura lift $\Phi_{L,3}(f_{2/3}\otimes \mathcal{F},z)$ by the cusp form $\eta(\tau)^6 \eta(3\tau)^6 \in S_{6}(3)$ and multiply it by the third powers of $j_3(z)-j_3(\alpha)$ for the four CM points $\alpha$ above, we obtain a weakly holomorphic modular form of weight $0$ for $\Gamma_0(3)$, which can hence be written as as a polynomial in the Hauptmodul $j_3(z)$. The necessary computation is tedious to do by hand, but easily done in a computer algebra system. The result is that the Shimura lift may be written as
\[
\frac{1}{2}(-6i)^{-3}\Phi_{L,3}(f_{2/3}\otimes \mathcal{F},z)  = \frac{1}{2}\eta(z)^6 \eta(3z)^6 \frac{A(j_3(z))}{B(j_3(z))}
\] 
with the degree $12$ polynomials
\begin{align*}
A(X) &= -64X^{12} - 7660X^{11} - 516744X^{10} - 24576796X^9 - 987878688X^8  \\
&\quad - 31994374680X^7 - 884044074800X^6 - 23323899141720X^5  \\
&\quad - 524999237829408X^4 - 9521554324373244X^3 - 145943768399337864X^2 \\
&\quad - 1577126071845011340X - 9606056659007943744
\end{align*}
and
\begin{align*}
B(X) &= (X^2 - 10X + 729)^3 (X^2 + 46X + 729)^3.
\end{align*}
We remark that $A(X)$ is irreducible over $\Q$ and $A(X),B(X)$ are coprime. Now we may plug in the CM point $z_U$ and find, using the evaluation 
\[
j_3(z_U) = -27,
\]
that
\[
\frac{1}{2}(-6i)^{-3}\Phi_{L,3}(f_{2/3}\otimes \mathcal{F},z_U) = \frac{183\sqrt{3}}{2}\Omega_{-3}^6.
\]
Taking everything together, we obtain the evaluation
\[
\langle f_{2/3},\Theta_{P,4}\rangle_{\Pet} = \frac{1}{12^4}\Phi_{L,3}(f_{2/3} \otimes \mathcal{F},z_U) =  \frac{61\sqrt{3}i}{32}\Omega_{-3}^6.
\]

We can compute some more coefficients in this way, and after dividing by $\|\Theta_{P,4}\|_{\Pet}^2 = \frac{1}{16}\Omega_{-3}^6$ we obtain the holomorphic part
\begin{align*}
\tilde{\Theta}_{P,4}^+(z) =& -\frac{\sqrt{3}i}{2} (\phi_{1}-\phi_{-1}) \\
& \times \left(q^{-\frac{1}{3}} -61q^{\frac{2}{3}} - \frac{65804}{5^3}q^{\frac{5}{3}}- \frac{1566912}{8^3}q^{\frac{8}{3}}-\frac{19145526}{11^3}q^{\frac{11}{3}}-\frac{159360544}{14^3}q^{\frac{14}{3}}\dots\right)
\end{align*}
of a harmonic Maass form which maps to $\Theta_{P,4}/\|\Theta_{P,4}\|_{\Pet}^{2}$ under $\xi_{-2}$.

This also yields a $\xi$-preimage of the CM newform $g=\eta(3z)^8$. We easily find that 
\begin{align}\label{our holomorphic part}
\tilde{g}^+(z) = \frac{1}{4}\left(q^{-1} -61q^{2} - \frac{65804}{5^3}q^{5}- \frac{1566912}{8^3}q^{8}-\frac{19145526}{11^3}q^{11}-\frac{159360544}{14^3}q^{14}\dots\right)
\end{align}
is the holomorphic part of a harmonic Maass form which maps to $g/\|g\|_{\Pet}^{2}$ under $\xi_{-2}$. We would like to compare this to the holomorphic part
\begin{align}\label{bor holomorphic part}
q^{-1} -\frac{1}{4}q^2 + \frac{49}{5^3}q^5 - \frac{48}{8^3}q^8 - \frac{771}{11^3}q^{11}+ \frac{2744}{14^3}q^{14}+\dots
\end{align}
of a $\xi$-preimage of $g/\|g\|_{\Pet}^{2}$ constructed by Bruinier, Ono, and Rhoades in \cite{bruinieronorhoades}. The difference between the two harmonic Maass forms (i.e. the difference between the two holomorphic parts in \eqref{bor holomorphic part} and \eqref{our holomorphic part}) is given by
\[
\frac{3}{4}\left(q^{-1}+20q^2 + 176q^5 + 1020 q^8 + 4794q^{11} + 19360q^{14}+\dots\right).
\]
Indeed, we have the weight $-2$ weakly holomorphic modular form of level $9$,
\begin{align*}
&\frac{\eta(3z)^2}{\eta(z)^3 \eta(9z)^3}-3\frac{\eta(z)^3\eta(9z)^3}{\eta(3z)^{10}}-18\frac{\eta(9z)^6}{\eta(3z)^{10}}  \\
&\quad = q^{-1}+20q^2 + 176q^5 + 1020 q^8 + 4794q^{11} + 19360q^{14}+\dots
\end{align*}
The reason that our $\xi$-preimage differs from the one constructed in \cite{bruinieronorhoades} by a weakly holomorphic modular form is the fact that the components of vector-valued modular forms typically have non-vanishing principal parts at many different cusps, so they are not \emph{good} in the sense of \cite{bruinieronorhoades}. Hence, to obtain a good form, one needs to subtract a suitable weakly holomorphic modular form from our holomorphic part.

\bibliography{references}
\bibliographystyle{alpha}

\end{document}